\documentclass[12pt]{amsart}
\usepackage{amsfonts, amssymb,amsmath, amsthm, amsxtra, latexsym, mathrsfs, amscd}
\usepackage{amsmath, amssymb, amsthm, times, float}
\usepackage{mathtools, amssymb, amsthm, times, float}
\usepackage{graphics}
\usepackage{MnSymbol} 
\usepackage{enumerate}
\usepackage{enumitem}
\usepackage{tikz-cd}

\usepackage{amssymb,hyperref}
\usepackage{backref}
\usepackage[latin1]{inputenc} %

\newtheorem{theo+}{Theorem}[section]
\newtheorem{prop+}[theo+]{Proposition}
\newtheorem{coro+}[theo+]{Corollary}
\newtheorem{lemm+} [theo+]{Lemma}
\newtheorem{deep+}  [theo+]  {Deep Result}
\newtheorem{fact+}  [theo+]  {Fact}

\theoremstyle{definition}
\newtheorem{exam+}  [theo+]  {Example}
\newtheorem{rema+}  [theo+]  {Remark}
\newtheorem{defi+}  [theo+]  {Definition}
\newtheorem*{definition*}{Definition}
\newtheorem{xca+}[theo+]{Exercise}
\newtheorem{thmalph}{Theorem}

\newenvironment{theorem}{\begin{theo+}}{\end{theo+}}

\newenvironment{corollary}{\begin{coro+}}{\end{coro+}}
\newenvironment{lemma}{\begin{lemm+}}{\end{lemm+}}

\newenvironment{remark}{\begin{rema+}}{\end{rema+}}

\newtheorem{conj+}[theo+]{Conjecture}

\numberwithin{equation}{section}

\usepackage{color}

\newcommand\beq{\begin{equation}\label}
\newcommand\eeq{\end{equation}}

\renewcommand\a[1]{{\acute{#1}}}

\def\draft{\centerline{(Draft {\the \day}/{\the\month} \the \year.)}}
\bigskip

\def\refn#1.#2{\expandafter\def\csname#1\endcsname{[#2]}}
\def\refnr#1.{\csname#1\endcsname}

\def\fa{\mathfrak a}

\def\fg{\mathfrak g}

\def\fk{\mathfrak k}
\def\fh{\mathfrak h}
\def\fl{\mathfrak l}
\def\fm{\mathfrak m}

\def\fp{\mathfrak p}

\def\ft{\mathfrak t}

\def\a{\alpha}

\def\Claminv2{|C(\Lambda)|^{-2}}

\def\varepsi{\varepsilon}

\def\de{d\varepsilon}

\def\Aa2D{A^{\a,2}(D)}
\def\bAa2D{\overline{A^{\a,2}(D)}}
\def\Ab2D{A^{\beta,2}(D)}
\def\bAb2D{\overline{A^{\beta,2}(D)}}

\def\abs#1{\vert#1\vert}

\def\Norm#1_#2{\Vert#1\Vert_{#2}}

\def\phipl12{\phi_{p_{l_1}, p_{l_2}}}
\def\phip01{\phi_{p_{0}, p_{0}}}

\def\a{\alpha}

\def\Claminv2{|C(\Lambda)|^{-2}}

\def\varepsi{\varepsilon}

\def\ad{\operatorname{ad}}
\def\Ad{\operatorname{Ad}}

\def\exp{\operatorname{exp}}
\def\Exp{\operatorname{Exp}}

\def\de{d\varepsilon}

\def\Aa2D{A^{\a,2}(D)}
\def\bAa2D{\overline{A^{\a,2}(D)}}
\def\Ab2D{A^{\beta,2}(D)}
\def\bAb2D{\overline{A^{\beta,2}(D)}}

\def\phipl12{\phi_{p_{l_1}, p_{l_2}}}
\def\phip01{\phi_{p_{0}, p_{0}}}

\def\alg/{algebra}

\def\Alg/{Algebra}

\def\alt/{alternative} 
\def\anal/{analytic}
\def\analfunc/{\anal/\ \func/}
\def\Ans/{\it Answer. \normal}
\def\ass/{associative}
\def\nass/{non-\ass/}
\def\autom/{automorphism}
\def\homom/{homomorphism}
\def\isom/{isomorphism}
\def\bdd/{bounded}
\def\Bdd/{Bounded}
\def\bddsymdom/{bounded \sym/ \dom/}
\def\Cartdom/{Cartan \dom/}
\def\bdry/{boundary}
\def\bsd/{\bdd/ \symdom/}
\def\bv/{boundary value}
\def\cf/{{\it cf}\.}
\def\Cf/{{\it Cf}\.}
\def\charr/{character}
\def\coeff/{coefficient}
\def\comm/{commutative}
\def\cpct/{compact}
\def\compl/{complex}
\def\comp/{complex}
\def\Comp/{Complex}
\def\conf/{conformal}
\def\conj/{conjugate}
\def\conn/{connect}
\def\cont/{continuous}
\def\conv/{converge} 
\def\convc/{convergence}
\def\convt/{convergent}
\def\convx/{convex}
\def\coord/{coordinate}
\def\lcoord/{local coordinate}
\def\Corr/{Corresponding}
\def\corr/{corresponding}
\def\corrd/{correspond}
\def\cov/{covariant}
\def\decomp/{decomposition}
\def\deco/{decompose}
\def\diff/{different} 
\def\Diff/{Different} 
\def\dimn/{dimension} 
\def\distr/{distribution} 
\def\div/{diverge} 
\def\dom/{domain}
\def\eg/{\hbox{\it e.g}\.}
\def\eigenf/{eigen\-\func/}
\def\eigensp/{eigen\-space}
\def\eigenv/{eigen\-value}
\def\eq/{equation}
\def\equa/{equation}
\def\de/{\diff/ial \equa/}
\def\do/{\diff/ial operator}
\def\ode/{ordinary \de/}
\def\pde/{partial \de/}
\def\pdo/{partial \diff/ial operator}
\def\psdo/{pseudo \diff/ial operator}
\def\fin/{finite}
\def\Ex/{\it Example.\ \normal}
\def\Exnr#1/{\it Example #1.\ \normal}
\def\foll/{follow}
\def\follg/{following}
\def\Follg/{Following}
\def\func/{function}
\def\Func/{Function}
\def\Fonc/{Fonc\-tion}
\def\fonc/{fonc\-tion}
\def\Funk/{Funk\-tion}
\def\funk/{Funk\-tion}
\def\gen/{general}
\def\har/{harmonic}
\def\Hint/{\it Hint. \normal}
\def\hist/{historic}
\def\histcl/{historical}
\def\hol/{holo\-morphic}
\def\homog/{ho\-mo\-ge\-ne\-ous}
\def\hyp/{hyper\-bolic}
\def\hyperg/{hyper\-geometric}
\def\ie/{\hbox{\it i.e.}}
\def\iff/{if and only if}
\def\ineq/{inequality}
\def\infra/{{\it inf\-ra}}
\def\ultra/{{\it ult\-ra}}
\def\Inpart/{In particular}
\def\inpart/{in particular}
\def\instof/{instead of}
\def\interps/{interpolation space}
\def\interp/{interpolation}
\def\Interp/{Interpolation}
\def\interpr/{Interpretation}
\def\Intr/{Introduction}
\def\intv/{interval}
\def\inv/{invariant}
\def\invc/{invariance}
\def\Iowords/{In other words}
\def\iowords/{in other words}
\def\ipr/{inner product}
\def\irred/{irreducible}
\def\lb/{line bundle}
\def\lin/{linear}
\def\lhs/{left hand side}
\def\rhs/{right hand side}
\def\loc/{local}
\def\math/{mathematic}
\def\mathcn/{\math/ian}
\def\manif/{manifold}
\def\meas/{measure}
\def\measl/{measurable}
\def\mero/{mero\-morphic}
\def\mon/{monomial}
\def\monog/{monogenic}
\def\mult/{multiple}
\def\multy/{multiply}
\def\multn/{multiplication}
\def\nas/{necessary and sufficient}
\def\nbd/{neighborhood}
\def\neg/{negative}
\def\nondeg/{nondegenerate}
\def\Oohand/{On the other hand}
\def\oohand/{on the other hand}
\def\Oonhand/{On the one hand}
\def\oonhand/{on the one hand}
\def\oper/{operator}
\def\orth/{ortho\-gonal}
\def\orthon/{ortho\-normal}
\def\otoh/{on the other hand}
\def\quat/{quaternion}
\def\pp/{\hbox{a. e.}}
\def\psh/{plurisubharmonic}
\def\pol/{polynomial}
\def\pot/{potential}
\def\pos/{positive}
\def\princ/{principle}
\def\prob/{probability}
\def\proj/{projective}
\def\projn/{projection}
\def\Proof/{\it Proof:\normal}
\def\Rem/{\it Remark\normal}
\def\Remnr#1/{\it Remark\ \normal #1. }
\def\rep/{representation}
\def\reps/{representations}
\def\meta/{metaplectic representation}
\def\repr/{reproducing}
\def\reprker/{reproducing kernel}
\def\resp/{respective} 
\def\resply/{respectively}
\def\restr/{restriction}
\def\sa/{self-adjoint}
\def\st/{such that}

\def\sol/{solution}
\def\ru/{space}
\def\sph/{spherical}
\def\ssp/{sub\ru/}
\def\sym/{symmetric}
\def\Sym/{Symmetric}
\def\symb/{symbol}
\def\symbc/{symbolic}
\def\symdom/{\sym/ domain}
\def\symp/{symplectic}
\def\Theor#1/{\fet Theorem #1.\ \normal}
\def\Lem#1/{\fet Lemma #1.\ \normal}
\def\Lemma/{\fet Lemma.\ \normal}
\def\topl/{topology}
\def\topll/{topological}
\def\transf/{transform}
\def\transl/{translation}
\def\transfn/{transformation}
\def\transv/{transvectant}
\def\trig/{trigonometric}
\def\tril/{trilinear}
\def\trilf/{trilinear form}
\def\uhp/{upper halfplane}
\def\uhs/{upper halfspace}
\def\vb/{vector bundle}
\def\vf/{vector field}
\def\vsp/{vector space}
\def\wrt/{with respect to}
\def\Wlog/{Without loss of generality}
\def\a{\alpha}

\def\Ab/{Abel}
\def\Ban/{Banach}
\def\Bansp/{\Ban/ space}
\def\Belt/{Bel\-tra\-mi}
\def\Berg/{Berg\-man}
\def\Bern/{Ber\-nou\-lli}
\def\Berz/{Berezin}
\def\Bess/{Bessel}
\def\Cart/{Car\-tan}
\def\Cay/{Cay\-ley}
\def\CG/{Clebsch-Gordan}
\def\Cl/{Clifford}
\def\CR/{Cauchy-Rie\-mann}
\def\Dir/{Dirichlet}
\def\Eucl/{Euclide}
\def\Eucln/{Euclidean}
\def\F/{Fourier}
\def\Hank/{Hankel}
\def\Hankf/{\Hank/ form}
\def\Herm/{Hermite}
\def\Hilb/{Hilbert}
\def\Hilbs/{Hilbert space}
\def\Hilbsp/{Hilbert space}
\def\HS/{Hilbert-Schmidt}
\def\Lag/{La\-grange}
\def\Lap/{La\-place}
\def\LapBelt/{\Lap/-\Belt/}
\def\Leb/{Lebesgue}
\def\Marc/{Mar\-cin\-kie\-wicz}
\def\Moeb/{Moebius}
\def\Moebt/{Moebius transformation}
\def\Moebtransfn/{Moebius transformation}
\def\Pla/{Plan\-che\-rel}
\def\Poin/{Poin\-car\'e}
\def\Riem/{Rie\-mann}
\def\Riemn/{\Riem/ian}
\def\psRiemn/{pseudo-\Riem/ian}
\def\Riems/{Rie\-mann surface}
\def\Schroe/{Schr\"odinger}
\def\Weier/{Weier\-strass}
\def\anal/{analytic}
\def\bsd/{bounded symmetric domain  }
\def\bdd/{bounded}
\def\calc/{calculation}\def\conj{conjugate}
\def\calci/{calculating}\def\eg{e.g.}
\def\conj/{conjugate}
\def\deco/{decomposition}
\def\eg/{e.g.}
\def\fct/{function}
\def\gp/{group}
\def\hw/{highest weight}
\def\hwv/{highest weight vector}
\def\hwvs/{highest weight vectors}
\def\lw/{lowest weight}
\def\lwv/{lowest weight vector}
\def\lwvs/{lowest weight vectors}
\def\hds/{holomorphic discrete series}
\def\iff/{if and only if}
\def\inv/{invariant}
\def\irrde/{irreducible decomposition}
\def\meas/{measure}
\def\transf/{transform}
\def\rep/{representation}
\def\resp/{respectively}
\def\inters/{intertwines}
\def\interg/{intertwining}
\def\meta/{metaplectic representation}
\def\qu/{quaternion}
\def\rep/{representation}
\def\symdom/{ symmetric domain}
\def\st/{such that}
\def\shd/{subhead}
\def\transf/{transform}
\def\wrt/{with respect to}

\def\Norm#1#2#3{\Vert#1\Vert^{#3}_{{#2}}}

\newmuskip\pFqmuskip

\newcommand*\pFq[6][8]{%
	\begingroup 
	\pFqmuskip=#1mu\relax
	\mathchardef\normalcomma=\mathcode`,
	\mathcode`\,=\string"8000
	\begingroup\lccode`\~=`\,
	\lowercase{\endgroup\let~}\pFqcomma
	{}_{#2}F_{#3}{\left(\genfrac..{0pt}{}{#4}{#5};#6\right)}%
	\endgroup
}
\newcommand{\pFqcomma}{{\normalcomma}\mskip\pFqmuskip}

\DeclareMathOperator{\gl}{\mathfrak{gl}}
\DeclareMathOperator{\SL}{SL}
\let\sl\relax
\DeclareMathOperator{\sl}{\mathfrak{sl}}

\DeclareMathOperator{\SO}{SO}
\DeclareMathOperator{\so}{\mathfrak{so}}
\DeclareMathOperator{\Sp}{Sp}

\let\sp\relax
\DeclareMathOperator{\sp}{\mathfrak{sp}}
\DeclareMathOperator{\SU}{SU}
\DeclareMathOperator{\su}{\mathfrak{su}}

\newcommand{\fraka}{\mathfrak{a}}

\newcommand{\frake}{\mathfrak{e}}
\newcommand{\frakf}{\mathfrak{f}}
\newcommand{\frakg}{\mathfrak{g}}
\newcommand{\frakh}{\mathfrak{h}}

\newcommand{\frakk}{\mathfrak{k}}
\newcommand{\frakl}{\mathfrak{l}}
\newcommand{\frakm}{\mathfrak{m}}
\newcommand{\frakn}{\mathfrak{n}}

\newcommand{\frakt}{\mathfrak{t}}
\newcommand{\fraku}{\mathfrak{u}}

\newcommand{\CC}{\mathbb{C}}

\newcommand{\RR}{\mathbb{R}}
\newcommand{\ZZ}{\mathbb{Z}}

\newcommand{\0}{\textbf{0}}

\DeclareMathOperator{\Hom}{Hom}

\renewcommand{\min}{{\textup{min}}}

\usepackage[edge-length=1cm,
fold-radius=1cm,
root-radius=.1cm]{dynkin-diagrams}

\begin{document}
	
\title
[Cartan--Helgason theorem for quaternionic
symmetric spaces]{Cartan--Helgason theorem for quaternionic symmetric and twistor spaces}

\begin{abstract} 
  Let $(\mathfrak g, \mathfrak k)$ be
 a complex quaternionic symmetric pair
 with $\fk$
 having an ideal $\mathfrak{sl}(2, \mathbb C)$,
 $\fk=\mathfrak{sl}(2, \mathbb C)+\fm_c$.
Consider the representation $S^m(\mathbb C^2)=\mathbb C^{m+1}$
of
$\fk$ via the projection onto the ideal
 $\fk\to \mathfrak{sl}(2, \mathbb C)$.
 We study  the finite dimensional irreducible representations
 $V(\lambda)$
 of $\mathfrak g$ which contain
 $S^m(\mathbb C^2)$ under $\fk\subseteq \fg$.
  We give a characterization of all such
  representations
  $V(\lambda)$
  and find the corresponding
  multiplicity
  $m(\lambda,m)=\dim \Hom (V(\lambda)|_\frakk,S^m(\CC^2)).$
  We consider also the branching problem of $V(\lambda)$ under 
$\fl=\mathfrak{u}(1)_{\mathbb C} +
  \mathfrak{m}_c\subset
  \frakk$ and find the multiplicities.
  Geometrically the 
 Lie subalgebra $\fl\subset \fk$
defines  a twistor space 
  over the compact symmetric space 
  of the compact real form $G_c$
  of $G_{\mathbb C}$, $\text{Lie}(G_{\mathbb C})=\fg$, and
  our results give the  decomposition
  for the $L^2$-spaces of sections of certain
  vector bundles
  over the symmetric space
  and line bundles over the twistor space.
  This generalizes Cartan--Helgason's theorem
  for symmetric spaces
  $  (\mathfrak g, \mathfrak k)$
  and Schlichtkrull's theorem
  for Hermitian symmetric spaces
  where one-dimensional representations
  of $\fk$ are considered.
\end{abstract}
\keywords{Cartan--Helgason theorem, quaternionic symmetric space, twistor space, finite dimensional representations, branching rule, multiplicity}

\subjclass[2020]{17B10, 22E46, 43A77, 43A85}

\author{Clemens Weiske}
\address{Mathematical Sciences, Chalmers University of Technology and Mathematical Sciences, G\"oteborg University, SE-412 96 G\"oteborg, Sweden}
\email{weiske@chalmers.se}

\author{Jun Yu}
\address{Beijing International Center for Mathematical Research and School of Mathematical Sciences,
	Peking University, No. 5 Yiheyuan Road, Beijing 100871, China}
\email{junyu@bicmr.pku.edu.cn}

\author{Genkai Zhang}
\address{Mathematical Sciences, Chalmers University of Technology and Mathematical Sciences, G\"oteborg University, SE-412 96 G\"oteborg, Sweden}
\email{genkai@chalmers.se}
\thanks{The first named author
  was supported by a research grant from the Knut and Alice Wallenberg
  foundation (KAW 2020.0275). The second named author was supported partially by the NSFC Grant 11971036. The third named author was supported
  partially by the Swedish Research Council (VR, Grants 2018-03402,  2022-02861).}

\maketitle


\section*{Introduction}

\subsection*{Cartan-Helgason Theorem}

In the present paper we shall prove a generalization of the Cartan-Helgason theorem
for quaternionic symmetric and twistor spaces. 
Let $G_c$ be a simply connected and simple compact Lie group
and $G_c/K$ a symmetric space of $G_c$.
The   Cartan-Helgason theorem
gives a parametrization of the representations
of $G_c$ appearing in the space $L^2(G_c/K)$
and thus a decomposition of the space $L^2(G_c/K)$.
It is commonly reformulated in
Lie algebraic terminology. So let
$(\mathfrak g, \mathfrak k)$ be a complex symmetric pair
with a Cartan decomposition $\mathfrak g=\mathfrak p +\mathfrak k$.
Let $\mathfrak a\subset \mathfrak p$ be a Cartan subspace
and $\mathfrak m$ the centralizer of  $\mathfrak a$ in $\frakk$,
$\mathfrak h=\mathfrak a+\mathfrak h_{\fm}\subseteq
\mathfrak p+\mathfrak k =\mathfrak g$ be a Cartan subalgebra of $\fg$
containing $\mathfrak a$.
In this setting, the Cartan-Helgason theorem \cite[Ch.~V]{Hel00}
gives  necessary and sufficient
condition for the highest weight $\lambda$
of an irreducible finite dimensional representation $V(\lambda)$
so that $V(\lambda)$ contains a non-zero
$\fk$-invariant vector. The condition
can briefly be stated as the restriction of $\lambda$
on $\fh_m$ is vanishing and the Killing form
of $\lambda$ against
the roots of $\fg$ with respect to $\fa$
satisfies certain integral and dominant
condition; moreover when the condition is satisfied
the multiplicity of the trivial representation
is one.  When $(\mathfrak g, \mathfrak k)$
is the complexification of a Hermitian
symmetric pair $(\mathfrak g_c, \mathfrak k_0)$
the Lie algebra $\mathfrak k_0$ has
also one-dimensional center and a one-parameter
family of one-dimensional representations.
Schlichtkrull \cite{Sch84} has found a generalization
for the Cartan-Helgason theorem characterizing
irreducible representations $V(\lambda)$ of $\fg$
containing a fixed one-dimensional $\fk$-representation.
It thus gives a decomposition
of $L^2(G_c/K, \chi)$ for the
$L^2$-space of sections of the
line bundle defined by a one-dimensional
character $\chi$.

We shall consider the case
when $(\fg, \fk)$ is a quaternionic symmetric pair.

\subsection*{Quaternionic symmetric
	and twistor spaces}

A compact quaternionic symmetric space $G_c/K$
is a symmetric space with $K$ being isomorphic to $\SU(2)\times M_c$
up to finite cover.
Each complex simple
Lie algebra has a compact  real form $\fg_c$
with $(\fg_c, \fk_0)$ being quaternionic symmetric with
$\fk_0$ containing an ideal $\su(2)$. Let $\fg$
be a complex simple Lie algebra with $\ft$ a Cartan subalgebra.
Let $\beta$ be the highest root. It determines
a subalgebra $\mathfrak{sl}(2, \mathbb C)
=\mathfrak{sl}(2, \mathbb C)_\beta$
and a Cartan involution
so that $(\fg, \fk)$ is a symmetric pair with $\fk=\mathfrak{sl}(2,\CC)
+\fm_c$. Let $\beta^\vee$ be the co-root of $\beta$.
Let ${\rm U}(1)\subseteq \SU(2)$
be such that the Lie algebra of ${\rm U}(1)$
is $\mathbb Ri\beta^\vee$. Then the space $G_c/{\rm U}(1)\times M_c$ is called 
a twistor space and has a complex structure 
defined by the positive root spaces of $\beta^\vee$. 

Let $\Gamma_m:=S^m(\mathbb C^2)$
be the irreducible representation of
$\mathfrak{sl}(2, \mathbb C)_\beta$ with highest weight $\frac m2
\beta$. We shall find a characterization of
the highest weight $\lambda$ of an irreducible
representation $ V(\lambda)$ of $\fg$ containing
the finite dimensional representation
$\Gamma_m\boxtimes 1_{\fm_c}$. As an application we
find the irreducible representation of
space $L^2(G_c/\SU(2)\times M_c, \Gamma_m\boxtimes 1_{\fm_c})$
of sections of the vector bundle 
and of $L^2(G_c/{\rm U}(1)\times M_c)$; see Remark \ref{final-rem}.

\subsection*{Main theorems}

We recall first the Cartan--Helgason theorem.
\begin{thmalph}[\cite{Hel00} Chapter V Theorem~4]
	Let $\lambda$ be an integral dominant weight and $V(\lambda)$ the irreducible finite dimensional representation of $\frakg$ of highest weight $\lambda$. The representation $V(\lambda)$ contains a $\frakk$-fixed vector if and only if $\lambda|_{\frakh_\frakm}=0$ and $\lambda(\alpha^\vee)\equiv 0 \pmod 2$ for every real simple root $\alpha$. In this case the subspace of $\frakk$-fixed vectors is one-dimensional.
\end{thmalph}

In the rest of the paper we let $(\fg, \fk)$, $\fk=\mathfrak{sl}(2,\CC)+\fm_c$,
be an exceptional quaternionic symmetric pair
with highest root $\beta$ (see Table~\ref{tab:2}).
The classical cases have been studied before in more generality (see
e.g. \cite{vDP99,Kna01,HTW05})  and we give the precise statement in
the quaternionic and twistor setting for completeness in
Section~\ref{sec:classical}.
We shall fix a system $\{\alpha_1, \cdots, \alpha_r\}$
of simple roots
and parametrize
the set $\Lambda$ of dominant weights
in terms of the corresponding fundamental weights $\{\varpi_j\}$;
see  (\ref{lambda-param}) below.
The irreducible representation
of  $\fg$ with a
highest weight $\lambda \in \Lambda$
will be denoted by $V(\lambda)$ throughout
the rest of the paper; similarly the representation
of $\fk$ with a highest weight  $\mu\in  \Lambda_{\fk}$ will
be denoted by $W(\mu)$ 
for $\mu\in  \Lambda_{\fk}$ with the simple
roots of $\fk$ being specified explicitly using
the Satake diagram.
Let $\alpha_1$ be the unique simple root connected to $-\beta$ in the affine Dynkin diagram and $\alpha_2$ the unique simple root connected to $\alpha_1$.
Then both $\alpha_1$ and $\alpha_2$ are real (see Table~\ref{tab:satake}).
For  $\lambda \in \Lambda$
we define
$$\lambda_1:=\lambda(\alpha_1^\vee), \qquad\lambda_2:=\lambda(\alpha_2^\vee).$$
For $d\in \ZZ_{\geq 0}$ let  $$\epsilon_i:=\begin{cases}
	0 & \text{if $\lambda_i-d\equiv 0 \pmod 2$,}\\ 
	1&\text{if $\lambda_i-d\equiv 1 \pmod 2$,} 
\end{cases} \qquad b_i:=\min\{\lambda_i,d-\epsilon_i\}.$$

\begin{thmalph}[See Theorem~\ref{thm:mult_not_G2}]
	Let $\frakg\not\cong \frakg_2$.
	The representation $V(\lambda)|_\frakk$ contains the representation $\Gamma_{m} \boxtimes \mathbf{1}_{\frakm_c}$ if only if
	\begin{enumerate}[label=(\roman*)]
		\item  $m=2d$ is even,
		\item$\lambda|_{\frakh_\frakm}=0$,
		\item $b_1+b_2 \geq d$,
		\item$\lambda(\alpha^\vee)\equiv 0 \pmod 2$ for every real simple root $\alpha\neq \alpha_1,\alpha_2$.
	\end{enumerate}
	In this case $V(\lambda)|_\frakk$ contains $\Gamma_{2d}\boxtimes \mathbf{1}_{\frakm_c}$ with multiplicity
	$$m(\lambda,d)=\left[\frac{b_1+b_2-d+2}{2}\right].$$
\end{thmalph}
We give a precise necessary and sufficient condition on the highest weight $\lambda$ such that $\lambda|_{\frakh_\frakm}=0$ in Remark~\ref{remark:h_m_res_0}.
Let $\frakl=\CC \beta^\vee \oplus \frakm_c \subseteq \frakk$.
\begin{thmalph}[See Theorem~\ref{thm:twistor_mult_not_G2}]
	Let $\frakg\not\cong \frakg_2$.
	Let $\lambda\in \Lambda$. The representation $V(\lambda)$ contains a $\frakl$-fixed vector if and only if 
	$\lambda|_{\frakh_\frakm}=0$,
	$\lambda(\alpha^\vee)\equiv 0 \pmod 2$ for every simple root $\alpha \neq \alpha_1,\alpha_2$.
	In this case the subspace of $\frakl$-fixed vectors in $V(\lambda)$ has dimension
	$$
	m(\lambda)=\left[\frac{(\lambda_1+1)(\lambda_2+1)+1}{2}\right].
	$$
\end{thmalph}

For $\frakg=\frakg_2$ we have  $\frakk= \sl(2,\CC)\oplus \frakm_c=\sl(2,\CC)_\beta\oplus \sl(2,\CC)_{\alpha_2}$.
We  consider $\frakk$-representations of the form $\Gamma_m \boxtimes \mathbf{1}_{\frakm_c}$ and also $\mathbf{1}_{\beta}\boxtimes \Gamma_m$, $\mathbf{1}_\beta$ being the trivial $\sl(2,\CC)_\beta$-module. Note that in our convention the left copy of $\sl(2,\CC)$ belongs to the highest root $\beta$. We let $\lambda_i, \epsilon_i$ and $b_i$ be as before and additionally we define
$$\lambda_2'=\begin{cases}
	\left[\frac{\lambda_2}{3}\right] & \text{if $\lambda_2 \equiv 0,2 \pmod 3$,}\\
	\left[\frac{\lambda_2}{3}\right]-1 & \text{if $\lambda_2 \equiv 1 \pmod 3$,}
\end{cases}, $$
$$\epsilon_2':=\begin{cases}
	0 & \text{if $\lambda'_2-d\equiv 0 \pmod 2$,}\\ 
	1&\text{if $\lambda_2'-d\equiv 1 \pmod 2$,} 
\end{cases} \qquad b_2':=\min\{\lambda_2',d-\epsilon_2'\}.$$
\begin{thmalph}[See Theorem~\ref{thm:mult_G2_short} and Theorem~\ref{thm:mult_G2_long}]
	Let $\frakg\cong\frakg_2$ and  $\lambda\in \Lambda$.
	\begin{enumerate}[label=(\arabic*)]
		\item $V(\lambda)|_\frakk$ contains the representation $\mathbf{1}_{\beta}\boxtimes \Gamma_m$ if only if
		\begin{enumerate}[label=(\roman*)]
			\item  $m=2d$ is even,
			\item $b_1+b_2 \geq d$,
		\end{enumerate}
		In this case $V(\lambda)|_\frakk$ contains $\mathbf{1}_{\frakm_c}\boxtimes \Gamma_{2d}$ with multiplicity
		$$m_{\alpha_2}(\lambda,d)=\left[\frac{b_1+b_2-d+2}{2}\right].$$
		\item $V(\lambda)|_\frakk$ contains the representation $\Gamma_{m} \boxtimes \mathbf{1}_{\frakm_c}$ if only if
		\begin{enumerate}[label=(\roman*)]
			\item  $m=2d$ is even,
			\item $b_1+b'_2 \geq d$,
		\end{enumerate}
		In this case $V(\lambda)|_\frakk$ contains $\Gamma_{2d}\boxtimes \mathbf{1}_{\frakm_c}$ with multiplicity
		$$m_\beta(\lambda,d)=\left[\frac{b_1+b'_2-d+2}{2}\right].$$
	\end{enumerate}
\end{thmalph}

Let $\frakl_\beta:=\CC\beta^\vee\oplus \sl(2,\CC)_{\alpha_2}$ and $\frakl_{\alpha_2}:=\sl(2,\CC)_\beta\oplus \CC {\alpha_2^\vee}$.
\begin{thmalph}[See Theorem~\ref{thm:twistor_mult_G2}]
	\begin{enumerate}
		\item Let $\lambda\in \Lambda$.
                  The representation $V(\lambda)$ contains a $\frakl_{\alpha_2}$-fixed vector.
		The space of $\frakl_{\alpha_2}$-fixed vectors has dimension $m_{\alpha_2}(\lambda)$, given by 
		$$
		m_{\alpha_2}(\lambda)=\left[\frac{(\lambda_1+1)(\lambda_2+1)+1}{2}\right].
		$$
		\item 	Let $\lambda\in \Lambda$. The representation $V(\lambda)$ contains a $\frakl_\beta$-fixed vector if and only if $\lambda_2\neq 1$.
		In this case the space of $\frakl_\beta$-fixed vectors has dimension $m_\beta(\lambda)$, given by 
		$$
		m_\beta(\lambda)=\left[\frac{(\lambda_1+1)(\lambda'_2+1)+1}{2}\right].
		$$
	\end{enumerate}
For the corresponding statements in the classical cases we refer the reader to Theorem~\ref{thm:Knapp_U}, Theorem~\ref{thm:twistor_mult_U} for case of unitary groups, Theorem~\ref{thm:Knapp_O}, Theorem~\ref{thm:twistor_mult_O} for the case of orthogonal groups and  Theorem~\ref{thm:Knapp_Sp}, Corollary~\ref{cor:twistor_mult_Sp} for the case of symplectic groups.
\end{thmalph}

\subsection*{Methods}
The Cartan-Helgason theorem has basically two kinds
of proofs, analytic and algebraic; see e.g.
\cite[Theorem 4.1, p.~535]{Hel00}, \cite[Theorem 8.4, p.~544]{Kna02},
\cite{Kos04}.
The analytic proof \cite{Hel00}
related
the problem about finding $K$-invariant
vectors in representations
of $G_c$ to the induced representations
of a real form $G$ of $G_{\mathbb C}$. More precisely
let $G/K$ be a non-compact symmetric
space with $\fg_0$
a real form of $\fg$.
Let $\fg_0=\fp_0+\fk_0$
be the Cartan decomposition and
$\fa_0\subset \fp_0$ a Cartan subalgebra
of $\fp_0$. Let $KAN$ be
the Iwasawa decomposition of $G$
and $MAN$ the minimal parabolic
subgroup. Consider the
induced representation of $G$
from $MAN$ with  a generic  character of $A$.
Realizing the induced representation
as functions on $G/K$ the $K$-invariant
element becomes the Harish-Chandra
spherical function. Given a finite dimensional
representation $V(\lambda)$ of $G$
for a dominant $\lambda$ the analytic proof starts with a construction
of $K$-invariant vector by integrating
the highest weight vector over $K$.
That such vector is non-zero is
proved by using the non-vanishing
of the Harish-Chandra $c$-function
at the dominate points $\lambda$;
see loc. cit.. The similar proof
is used in \cite{Sch84} for Hermitian
symmetric pairs.

The algebraic proof and its generalization
by Kostant \cite{Kos04} is roughly speaking
a  reformulation of the induced representation
and by parametrization of the representation
space as a quotient space of the universal
enveloping algebra of $\fk$.
The problem of finding the multiplicity
of $(W(\mu), \fk)$ in
$(V(\lambda), \fg)$ is 
reduced to computing certain projections
of root vectors and solving
some algebraic equations in the given
representation space  $(W(\mu), \fk)$.
We shall apply Kostant's theorem and
find a precise multiplicity formula.

\subsection*{Further questions and related results}

We remark that there is an important
class of infinite dimensional representations, called
minimal representations, which have
been studied extensively in particular
for their branching. For quaternionic
non-compact symmetric pairs $(\fg_0, \fk_0)$
Gross and Wallach \cite{GW96} have found the minimal representations
as continuation of quaternionic discrete series.
In a certain sense our finite dimensional
representations are finite dimensional  versions
of  quaternionic discrete series and it might
be interesting to study  branching problems for this
family of representations under
subgroups.
As mentioned above our results are naturally
related to induced representations and we plan
to study them in future.
We note that when $(\fg, \fk)$ is rank one some more 
analytic results have been studied in 
\cite{Cam05} and \cite{vDP99}. For classical symmetric pairs
$(G, K_2\times K_1)=
({\rm U}(n+m), {\rm U}(n)\times {\rm U}(m)), 
(\SO(n+m), \SO(n)\times \SO(m)), (\Sp(n+m), \Sp(n)\times \Sp(m))$, with $n\le m$, Knapp
\cite{Kna01} has proved certain duality relations between
multiplicity spaces for
the  pairs $(G, K_2\times K_1)$ and
$(G', K')=({\rm U}(n)\times {\rm U}(n), {\rm U}(n)), 
(\SO(n)\times \SO(n), \SO(n)),
(\Sp(n)\times \Sp(n), \Sp(n))$, with $n\le m$;
more precisely let $(V, G)$ be an irreducible
representation of $G$ and $V^{K_1}$ be
the subspace of $K_1$-fixed vectors.
Then there exists a representation $(W, G')$
of $G'$ such that the two representations
$(V|_{K_2}, K_2)$ and $(W|_{K_2}, K_2)$
are equivalent. We can then deduce
our multiplicity formula for
classical symmetric pairs using
this duality; see Section \ref{sec:classical} below.
There are some general
Littlewood formulas
\cite{GW98} expressing 
multiplicities as sums of Littlewood-Richards coefficients.
Our main contribution to the classical cases is an exact formula
for the quaternionic symmetric pairs.

\subsection*{Organization of the paper}
The paper is organized as follows. In Section~\ref{sec:prelims} we recall some facts about quaternionic symmetric pairs $(\frakg,\frakk)$ and the real forms $\frakg_0$, as well as Kostant's theorem for the branching to symmetric subgroups \cite{Kos04}.
The proofs of the main theorems in Section~\ref{sec:proofs} use some explicit projections to the ideal $\sl(2,\CC)_\beta\subseteq \frakk$. These are found in Section~\ref{sec:exceptional_projections} and given using reduction to the case $\frakg_0=\so(4,4)$, which is discussed in Section~\ref{sec:SO(4,4)}.
In Section~\ref{sec:classical} we give a brief
account for the same branching problem
for classical quaternionic Lie groups;
most of these  results can be obtained 
from \cite{Kna01}. In the orthogonal case we apply the theorem of Kostant again to prove some explicit restriction formula for representations of ${\rm U}(4)$ to $\SO(4)$.
We include an Appendix, containing some linear algebraic results concerning some special matrices, which is used in the proof of the main theorems (Appendix~\ref{app:linear_algebra}) and relevant tables of quaternionic symmetric pairs, Dynkin diagrams and root datum (Appendix~\ref{app:tables}).

\subsection*{Acknowledgements}
We would like to thank Sangjib Kim for discussions
on  branching problems for  orthogonal groups.
We thank also Robin van Haastrecht for pointing out some incomplete arguments in
the proof of
Proposition \ref{prop:joint_kernel} in an earlier version of this paper, which has
led us to find the current  more conceptual proof. In the beginning of this project we have used the computer program 
LiE to compute the multiplicities for various examples.

\section{Preliminaries}\label{sec:prelims}
\subsection{The quaternionic real form}
We first fix some
general notation and conventions.
Let $G_\CC$ be a simply connected simple complex Lie group and $\frakg$ its Lie algebra. Let $G\subseteq G_\CC$ be a real non-compact
simple subgroup of $G_\CC$ with  Lie algebra $\frakg_0$.
Let $(\frakg_0, \frakk_0)$ be
the non-compact symmetric pair and 
$(\frakg, \frakk)$
the complexified symmetric pair with  Cartan involution $\theta$.
The Killing form
on $\frakg$ will be denoted by $B(\cdot, \cdot)$.
Let $G=KAN$ be the Iwasawa-decomposition of $G$ and  $\fraka,\frakn\subseteq \frakg$  the complexified Lie algebras of $A,N$. Let $M$ be the centralizer of $A$ in $K$ with complexified Lie algebra $\frakm$. 
Every complex simple Lie algebra has a unique quaternionic real form $\frakg_0$ i.e.  $\frakg$ has a unique (up to conjugation) Cartan-involution $\theta$, such that $\frakk=\sl(2,\CC)\oplus \frakm_c$ contains a simple ideal
$\sl(2,\CC)$ with its real form $\frakk_0=\mathfrak{su}(2)\oplus \frakm_0$. 
A list of
$\frakk\subseteq \frakg$
for exceptional simple Lie algebras is given in Table~\ref{tab:2}.
Note that the quaternionic real form is uniquely determined by its real rank, which is denoted as the additional index for $\frakg_0$ in the third column of Table~\ref{tab:2}.
The corresponding list for  classical Lie algebras is in Section~\ref{sec:classical}
along with the study of the branching problem.

 Let $\frakt \subseteq \frakg$ be a maximal compact Cartan subalgebra and let $\Delta(\frakg,\frakt)$ be the associated root system. Let $\Pi(\frakg,\frakt)=\{\delta_1,\dots,\delta_r\}$ be a choice of simple roots.
For any positive root $\delta$ let $\{X_{\delta},X_{-\delta}, H_{\delta}\}$ be the associated $\sl(2,\CC)$-triple, with $H_\delta \in \frakt$, $\delta(H_\delta)=2$ and $[X_\delta,X_{-\delta}]=H_\delta$.
  Let for any root $\delta\in \Delta(\frakg,\frakt)$  $$P_\delta:\frakk \to \sl(2,\CC)_\delta$$ be the parallel projection.
 Let $\beta \in \Delta(\frakg,\frakt)$ be the highest root such that $-\beta$ is the lowest root.
 
 The root system of the maximal subalgebra $\frakk$ is obtained from the affine Dynkin diagram of $\frakg$, which additionally contains the lowest root $-\beta$. In the exceptional cases this root is connected to a unique simple root. Deleting the unique vertex connected to $-\beta$ one obtains the root system of $\frakk$, where a copy of $\sl(2,\CC)$ is now obtained
 by the highest root $\beta$ by the general method. The affine Dynkin diagrams of the relevant cases can be found in Table~\ref{tab:satake}, see also \cite{GW96}.

\subsection{Cayley transform to the maximal split Cartan}
 Let $r_0=\dim_\CC \fraka$. Then $r_0$ is the real rank of $\frakg_0$. Let $\{\gamma_1,\dots, \gamma_{r_0}\}$ be a choice of strongly orthogonal non-compact roots. In particular we have that $\frakk=\sl(2,\CC)_\beta\oplus \frakm_c$.
 Let $$c_{\gamma_j}:=\Ad\left(\exp\left(\frac{\pi}{4}(X_{-\gamma_j}-X_{\gamma_j})\right)\right)$$
 be the Cayley transform associated with $\gamma_j$ and let
 $$c:=c_{\gamma_1}\circ \dots \circ c_{\gamma_{r_0}}$$ be the composition. Note that the transforms $c_{\gamma_i}$ commute, such that there is no need to choose an ordering of $\{\gamma_i\}$.
 
 Let $\frakh:=\fraka \oplus \frakh_\frakm$ be a maximal split Cartan subalgebra, such that $\frakh_\frakm\subseteq \frakt$ is a Cartan subalgebra of $\frakm$. In particular $\frakh=c(\frakt)$ and $c$ fixes $\frakh_\frakm \subseteq\frakt$. Then $c:\frakg \to \frakg$ is an automorphism of $\frakg$, mapping $\frakt$ to $\frakh$ and simple roots and root vectors of $(\frakg,\frakt)$ to simple roots and root vectors of $(\frakg,\frakh)$. 
 We define $$\alpha_i:=c(\delta_i), \qquad 1\leq i\leq r,$$
 such that $\Pi(\frakg,\frakh)=\{\alpha_1,\dots,\alpha_r\}$ is a choice of simple roots for $\Delta(\frakg,\frakh).$ We denote the coroots of $\alpha_j$ by $\alpha_j^\vee \in \frakh$. Note that for non-compact roots the standard $\sl(2,\CC)$-triple is then given
 by $\{\alpha_j^\vee, ic(X_{\delta_j}),ic(X_{-\delta_j})\}$. Moreover note that compact roots in $\Delta(\frakg,\frakh)$ are fixed by $c^{-1}$, such that both notations $\delta$ and $\alpha$ can be used.

 \subsection{Restricted roots and Satake diagrams}
Recall that a root $\alpha\in \Delta(\frakg,\frakh)$ is called \emph{real} if $\alpha|_{\frakh_\frakm}=0$, \emph{imaginary} if $\alpha|_\fraka=0$ and \emph{complex} otherwise. We define the restricted roots $$\bar{\alpha}:=\alpha|_\fraka.$$
The Satake diagram of $(\frakg,\frakh)$ is the Dynkin diagram, where vertices corresponding to imaginary roots are colored black and where vertices with the same restriction to $\fraka$ are joint by a double arrow. The Satake diagrams of the exceptional pairs $(\frakg,\frakh)$ associated to the quaternionic real form and the real roots are given in Table~\ref{tab:satake}.
The simple roots for $\frakm$ are precisely the black colored roots in the Satake diagram from which it is clear that $\frakm\subseteq\frakm_c$.

\subsection{Kostant's multiplicity formula for maximal subalgebras}
Let $\frakm=\frakm^-\oplus \frakh_\frakm \oplus \frakm^+$ be the triangular decomposition of $\frakm$. In particular we have that by our construction $$\frakm^+\oplus \frakn=\sum_{\alpha \in \Delta^+(\frakg,\frakh)}\frakg_\alpha.$$
For any $\alpha_i\in \Pi(\frakg,\frakh)$ let $\varpi_i$ be the corresponding fundamental weight,
such that we identify the set $\Lambda$ of integral dominant weights with $\ZZ_{\geq 0}^r$, via
\begin{equation}
  \label{lambda-param}
\ZZ_{\geq 0}^r \to \Lambda,  \qquad (\lambda_1,\dots,\lambda_r)\mapsto\sum_{i=1}^r \lambda_i\varpi_i. 
\end{equation}
For $\lambda\in \Lambda$, we denote by $V(\lambda)$ the irreducible highest-weight module of $\frakg$ with highest weight $\lambda$. 
We fix a choice of positive $(\frakk,\frakt\cap \frakk)$-roots and we denote the set of integral dominant weights of $\frakk$ by $\Lambda_\frakk$.
For $\mu\in \Lambda_\frakk$ let $W(\mu)$ be any irreducible finite dimensional $\frakk$-module with highest weight $\mu$.
By construction we have that $\lambda|_{\frakh_\frakm}$ is integral dominant for $\frakm$ with respect to $\frakm^+$. We denote by $W(\mu)[\lambda|_{\frakh_\frakm}]$ the isotypic component of highest weight $\lambda|_{\frakm}$ inside the restriction $W(\mu)|_\frakm$. We further denote by 
$W(\mu)[\lambda|_{\frakh_\frakm}]^{\frakm^+}$ the set of $\frakm^+$-highest weight vectors.
We define the multiplicity
$$m(\lambda,\mu)=\dim \Hom (V(\lambda)|_\frakk,W(\mu)).$$
Finally for any simple root $\alpha_j \in\Pi(\frakg,\frakh) \subseteq \Delta(\frakg,\frakh)$ we define
$$
Z_{\alpha_j}=c(X_{\delta_j})+\theta( c(X_{\delta_j}))\in \frakk.
$$

Let $$\Pi_n(\frakg,\frakh):=\{ \alpha_i \in \Pi(\frakg,\frakh)|\; \alpha_i \text{ is not imaginary}  \}.$$
The simple roots $\Pi_n$ can be read off  from Satake diagrams given in Table~\ref{tab:satake}, where the imaginary roots are painted black.
Note that under this normalization we have that $B(Z_{\alpha_j},Z_{\alpha_j})=B(\alpha_j^\vee,\alpha_j^\vee)$ whenever $\alpha_j\in \Pi_n$.
We recall Kostant's theorem on the multiplicity $m(\lambda,\mu)$.
\begin{theorem}[\cite{Kos04} Theorem~0.5]\label{thm:Kostant}
	Let $\lambda \in \Lambda$ and $\mu \in \Lambda_\frakk$.  We have
	$$m(\lambda,\mu)=\dim \{  v \in W(\mu)[\lambda|_{\frakh_\frakm}]^{\frakm^+}|\; q_{\lambda,i}(Z_{\alpha_i})v=0, \, \forall \alpha_i\in \Pi_n(\frakg,\frakh) \},$$
	where $q_{\lambda,i}$ is a polynomial given by
	$$
	q_{\lambda,i}(t)=\begin{cases}
		(t-\lambda_i)(t-(\lambda_i-2))\dots (t+\lambda_i)& \text{if $\alpha_i$ is a real simple root,}\\
		t^{\lambda_i+1} &\text{otherwise.}
	\end{cases}
	$$
\end{theorem}
From Kostant's theorem we can readily deduce the following in the exceptional quaternionic setting.
\begin{corollary}\label{cor:Kostant_quaternionic}
	Let $\lambda \in \Lambda$ and $\Gamma_m$ the $m+1$ dimensional irreducible representation of $\sl(2,\CC)$. Consider the $\frakk=\sl(2,\CC)\oplus \frakm_c$-representation $\Gamma_m\boxtimes \mathbf{1}_{\frakm_c}$.
	Then $V(\lambda)|_\frakk$ contains the representation $\Gamma_m\boxtimes \mathbf{1}_{\frakm_c}$ only if 
	\begin{enumerate}[label=(\roman*)]
		\item $\lambda|_{\frakh_\frakm}=0$,
		\item $\lambda(\alpha_j^\vee)\equiv 0 \pmod 2$ for all real simple roots $\alpha_j$ with 
		$P_\beta(Z_{\alpha_j})=0.$
	\end{enumerate}
	In this case the multiplicity is given by
	$$\dim \{ v \in \Gamma_m|\, q_{\lambda,i}(P_\beta(Z_{\alpha_i}))v =0, \, \forall \alpha_i \in \Pi_n(\frakg,\frakh) \}.$$
\end{corollary}
\begin{proof}
	We apply Theorem~\ref{thm:Kostant} to the setting of an exceptional complex quaternionic symmetric pair $(\frakg,\frakk)$.
    Since $\frakh_{\frakm}\subseteq\frakm\subseteq\frakm_c$ it is clear that $$(\Gamma_m\boxtimes \mathbf{1}_{\frakm_c})[\lambda|_{\frakh_{\frakm}}]=\Gamma_m\boxtimes (\mathbf{1}_{\frakm_c}[\lambda|_{\fh_\fm}])=\{0\}$$
	if $\lambda|_{\frakh_{\frakm}}\neq 0$. Now assume $\lambda|_{\frakh_\frakm}=0$. Then every element of $\Gamma_m\boxtimes \mathbf{1}_{\frakm_c}$ is a highest weight vector for $\frakm$, since $\frakm$ only acts on the trivial representation $\mathbf{1}_{\frakm_c}$ of $\frakm_c$.
	Now assume $P_\beta(Z_{\alpha_j})=0$ for some $\alpha_j\in \Pi_n(\frakg,\frakh)$, i.e $Z_{\alpha_j}\in \frakm_c$. Then
	$Z_{\alpha_j}v=0$ for each $v\in \Gamma_m\boxtimes\mathbf{1}_{\frakm_c}$, such that 
	$$p_{\lambda,j}(Z_{\alpha_j})v=p_{\lambda,j}(0)v,$$
	which vanishes if and only if $\alpha_j$ is complex or $\alpha_j$ is real and $\lambda_j$ is even.
	Now assume $P_\beta(Z_{\alpha_j})\neq 0$, such that $Z_{\alpha_j}=P_\beta(Z_{\alpha_j})+X$ with $X\in \frakm_c$.
	Then clearly for any $v\in \Gamma_m\boxtimes \mathbf{1}_{\frakm_c}$ we have 
	$$p_{\lambda,j}(Z_{\alpha_j})v=p_{\lambda,j}(P_\beta(Z_{\alpha_j}))v+RXv=p_{\lambda,j}(P_\beta(Z_{\alpha_j}))v$$
	for some polynomial $R$ of $P_\beta(Z_{\alpha_j}) $ and $X$. Hence  $p_{\lambda,j}(Z_{\alpha_j})v$
	 only depends on its projection to $\sl(2,\CC)_\beta$ on the $\Gamma_m$-factor.
\end{proof}
 \begin{lemma}\label{lemma:root_vector_projection}
	Let $\gamma\in \Delta(\frakg,\frakt)$ be a root. If $\delta_j
        \notin \langle \gamma_1, \dots, \gamma_{r_0},
        \gamma\rangle_\ZZ$, then        $$P_\gamma(Z_{\alpha_j})=0.$$
\end{lemma}
\begin{proof}
	Let $\delta=\delta_j$ be a simple $(\frakg,\frakt)$-root.
	By expanding the exponentials defining $c_{\gamma_i}$, it is easy to see that
	$$c(X_\delta)=\sum_{\delta'}a_{\delta'}X_{\delta'},$$ 
	for some scalars $a_{\delta'}$, where the sum is to be taken
        over $\delta'\in\delta+ \langle \gamma_1, \dots,
        \gamma_{r_0}\rangle_\ZZ\subseteq \Delta(\frakg,\frakt)$. In
        particular no contribution of Cartan elements $H_{\delta'}$
        occurs
        in the projection
        $P_\gamma(Z_{\alpha_j})$, since $0\notin \delta+ \langle \gamma_1, \dots, \gamma_{r_0}\rangle_\ZZ$. Then $$c(X_\delta)+\theta(c(X_\delta))=\sum_{\delta'}a_{\delta'}(X_{\delta'}+\theta X_{\delta'})$$
	and since $\gamma \notin \delta+ \langle \gamma_1, \dots, \gamma_{r_0}\rangle_\ZZ,$ we have
	$$P_{\gamma}(X_{\delta'}+\theta X_{\delta'})=0,$$
	for $\delta' \in \delta+ \langle \gamma_1, \dots, \gamma_{r_0}\rangle_\ZZ,$
	which proves the statement.
\end{proof}

\section{Explicit projections for $\SO(4,4)$}\label{sec:SO(4,4)}
The proof of the main theorems makes use of a reduction to the case $\frakg_0=\so(4,4)$.
Let $\frakg=\so(8,\CC)$ and $\frakg_0=\so(4,4)$ such that 
$\frakk=\so(4,\CC)\oplus \so(4,\CC)$. Let $\frakt\subseteq \frakk\subseteq\frakg$ be a maximal compact Cartan subalgebra and $\Pi(\frakg,\frakt)=\{\delta_1,\dots,\delta_4\}$ be a choice of simple roots. 
We keep our convention such that roots of a maximal compact Cartan $\frakt$ are denoted by $\delta$'s, while roots for a maximal split Cartan $\frakh$ are denoted by $\alpha$'s.
Let $\beta$ be the highest root. Then we have the following affine Dynkin-diagram,
\begin{equation}\label{eq:D4}
\dynkin[labels={-\beta,\delta_1,\delta_2, \delta_3, \delta_4}]D[1]4
\end{equation}
In the standard realization we have that $$\Delta(\frakg,\frakt)^+=\{\varepsi_i\pm\varepsi_j|\; i<j\leq4\},$$
$$\delta_1=\varepsi_1-\varepsi_2,\qquad \delta_2=\varepsi_2-\varepsi_3, \qquad \delta_3=\varepsi_3-\varepsi_4,\qquad \delta_4=\varepsi_3+\varepsi_4$$
and $\beta=\delta_1+2\delta_2+\delta_3+\delta_4=\varepsi_1+\varepsi_2$ is the highest root and is compact.
We choose the following non-compact strongly orthogonal roots $\{\gamma_1,\dots,\gamma_4\}$:
\begin{equation}\label{eq:D4_SOR}
	\gamma_1=\varepsi_2-\varepsi_3, \qquad \gamma_2=\varepsi_2+\varepsi_3, \qquad \gamma_3=\varepsi_1-\varepsi_4,\qquad \gamma_4=\varepsi_1+\varepsi_4.
\end{equation}

\begin{prop+}\label{prop:D4_projections}
	\begin{enumerate}
		\item 	We have
		$$P_\beta(Z_{\alpha_j})=
		\begin{cases}
			-\frac{1}{2}H_\beta & \text{if $j=2$,}\\
			\frac{1}{2}(X_\beta+X_{-\beta}) & \text{otherwise.}
		\end{cases}
		$$
		\item		We have for $i=1,3,4$
		$$P_{\delta_i}(Z_{\alpha_j})=
		\begin{cases}
			\frac{1}{2}H_{\delta_i} & \text{if $j=2$,}\\
			-\frac{1}{2}(X_{\delta_i}+X_{-\delta_i})& \text{if $i=j$,}\\
			0& \text{otherwise.}
		\end{cases}
		$$
	\end{enumerate}
\end{prop+}
\begin{proof}
	Ad (1):
	Since $\delta_2=\gamma_1$ is a strongly orthogonal non-compact root itself,
	$$c_{\gamma_j}(X_{\gamma_1})=X_{\gamma_1}$$ for all $j\neq 1$ and by \cite[(6.66)]{Kna02} that
	$$c_{\gamma_1}(X_{\gamma_1})=-\frac{1}{2}H_{\gamma_1}+\frac{1}{2}(X_{\gamma_1}-X_{-\gamma_1}).
	$$
	Since $\gamma_1$ is non-compact it follows that $c(X_{\delta_2})+\theta c(X_{\delta_2})=-H_{\gamma_1}$. Since $$\gamma_1=\frac{1}{2}(\beta-\delta_1-\delta_3-\delta_4)$$
	and $\beta,\delta_1,\delta_3,\delta_4$ are compact, this implies the statement for $\delta_2$.
	
	Now let $\delta=\delta_1$. Then it is easy to see that $$\ad(X_{-\gamma_1}-X_{\gamma_1})X_{\delta}=-X_{\delta+\gamma_1}, \qquad \ad^2(X_{-\gamma_1}-X_{\gamma_1})X_{\delta}=-X_{\delta},$$
	such that 
	\begin{align*}
		c_{\gamma_1}(X_{\delta})&=\sum_{j=0}^\infty \frac{1}{(2n)!}\left(\frac{\pi}{4}\right)^{2n}(-1)^nX_{\delta}+\sum_{j=0}^\infty \frac{1}{(2n+1)!}\left(\frac{\pi}{4}\right)^{2n+1}(-1)^nX_{\delta+\gamma_1}\\
		&=\cos\left(\frac{\pi}{4}\right)X_{\delta}+\sin\left(\frac{\pi}{4}\right)X_{\delta+\gamma_1}\\
		&=\frac{1}{\sqrt{2}}(X_{\delta}+X_{\delta+\gamma_1}).
	\end{align*}
First note that $\beta=\delta_1+\gamma_1+\gamma_2$. We  have then
$$\ad(X_{-\gamma_2}-X_{\gamma_2})X_{\delta+\gamma_1}=-X_{\delta+\gamma_1+\gamma_2}=-X_\beta, \qquad \ad^2(X_{-\gamma_2}-X_{\gamma_2})X_{\delta+\gamma_1}=-X_{\delta+\gamma_1},$$

$$\ad(X_{-\gamma_3}-X_{\gamma_3})X_{\beta}=X_{\beta-\gamma_3}, \qquad \ad^2(X_{-\gamma_3}-X_{\gamma_3})X_{\beta}=-X_{\beta},$$

$$\ad(X_{-\gamma_4}-X_{\gamma_4})X_{\beta}=X_{\beta-\gamma_4}, \qquad \ad^2(X_{-\gamma_4}-X_{\gamma_4})X_{\beta}=-X_{\beta}.$$
Hence by a similar calculation than before we have
$$\frac{1}{\sqrt{2}}P_\beta(c_{\gamma_2}\circ c_{\gamma_3}\circ c_{\gamma_4}(X_{\delta+\gamma_1}))=\frac{1}{4}X_\beta.$$

On the other hand we have $-\beta=\alpha_1-\gamma_3-\gamma_4$ and by the same argument we obtain
$$\frac{1}{\sqrt{2}}P_\beta(c_{\gamma_2}\circ c_{\gamma_3}\circ c_{\gamma_4}(X_{\delta}))=\frac{1}{4}X_{-\beta},$$
such that
$$P_\beta(c(X_\delta)+\theta c(X_{\delta}))=\frac{1}{2}(X_\beta+X_{-\beta}),$$
since $X_\beta \in \frakk$.
By the triality of $D_4$ the statement also follows for $\delta_3$ and $\delta_4$.

Ad (2):
	The proof is essentially the same as the proof of (1).
\end{proof}
\section{Projections for exceptional quaternionic Lie algebras}\label{sec:exceptional_projections}
As before we shall use the convention that roots of a maximal compact Cartan $\frakt$ are denoted by $\delta$, while roots for a maximal split Cartan $\frakh$ are denoted by $\alpha$.
\subsection{The case $\frake_6$}
Let $\frakg=\frake_6$ and $\frakg_0=\frake_{6,4}$ the real form of real rank $4$ in the notation of \cite{GW96}. Then $\frakk=\sl(2,\CC)\oplus \sl(6,\CC)$.
We choose the strongly orthogonal non-compact roots $\{\gamma_1,\dots,\gamma_4\}\in \Delta(\frakg,\frakt)$ given as linear combinations of simple roots $\Pi(\frakg,\frakt)=\{\delta_1,\dots,\delta_6\}$ by
$$
\gamma_1=\delta_2, \qquad \gamma_2=\begin{psmallmatrix}
& & 1& & \\
0&1&2&1&0
\end{psmallmatrix}, \qquad \gamma_3=\begin{psmallmatrix}
& & 1& & \\
1&1&2&1&1
\end{psmallmatrix}, \qquad \gamma_4=\begin{psmallmatrix}
& & 1& & \\
1&2&2&2&1
\end{psmallmatrix}.$$
The highest root is $\beta=\begin{psmallmatrix}
	& & 2& & \\
	1& 2&3&2&1
\end{psmallmatrix}$.
\begin{lemma}
	\begin{enumerate}[label=(\roman*)]\label{lemma:E6_D4}
	\item The span $\langle \gamma_1,\dots,\gamma_4 \rangle\cap \Delta(\frakg,\frakt)$ is a root system of type $D_4$ with simple roots
	$$\delta'_1=\delta_4, \qquad \delta'_2=\delta_2, \qquad\delta'_3=\begin{psmallmatrix}
		& & 0 & & \\
		0& 1&1&1&0
	\end{psmallmatrix}, \qquad \delta'_4=\begin{psmallmatrix}
		& & 0 & & \\
		1& 1&1&1&1
	\end{psmallmatrix}$$
and the same highest root $\beta$, see \eqref{eq:D4}.
\item We have $\delta_i \in \langle\gamma_1,\dots,\gamma_4, \beta \rangle_\ZZ$ if and only if $i=2,4$.
	\end{enumerate}
\end{lemma}
\begin{proof}
	Ad (i). We have $$\delta'_1=\frac{1}{2}(-\gamma_1+\gamma_2+\gamma_3-\gamma_4),\qquad
	\delta'_2=\gamma_1,$$ $$\delta'_3=\frac{1}{2}(-\gamma_1+\gamma_2-\gamma_3+\gamma_4),\qquad   \delta'_4=\frac{1}{2}(-\gamma_1-\gamma_2+\gamma_3+\gamma_4),$$
	and it is easily checked that $\delta'_i$ satisfy the right orthogonality relations.
	
	Ad (ii). This is easily deduced from the definitions of $\gamma_1,\dots, \gamma_4,\beta$.
\end{proof}
The lemma above allows us to use Lemma~\ref{lemma:root_vector_projection} and Proposition~\ref{prop:D4_projections} to deduce the following.
\begin{corollary}\label{cor:E6_projection}
	We have for simple roots $\alpha_j\in \Pi(\frakg,\frakh)$
	$$P_\beta(Z_{\alpha_j})=	\begin{cases}
		-\frac{1}{2}H_\beta & \text{if $j=2$,}\\
		\frac{1}{2}(X_\beta+X_{-\beta})& \text{if $j=4$,} \\
		0 & \text{otherwise.}
	\end{cases}$$
\end{corollary}

\subsection{The case $\frake_7$}
Let $\frakg=\frake_7$ and $\frakg_0=\frake_{7,4}$ the real form of real rank $4$. Then $\frakk=\sl(2,\CC)\oplus \so(12,\CC)$.
We choose the strongly orthogonal non-compact roots $\{\gamma_1,\dots,\gamma_4\}\in \Delta(\frakg,\frakt)$
in terms of the simple roots
$\Pi(\frakg,\frakt)=\{\delta_1,\dots,\delta_7\}$ by
$$
\gamma_1=\delta_1, \qquad \gamma_2=\begin{psmallmatrix}
	& & 1& & & \\
	1&2&2&1&0 &0
\end{psmallmatrix}, \qquad \gamma_3=\begin{psmallmatrix}
	& & 1& & &\\
	1&2&2&2&2&1
\end{psmallmatrix}, \qquad \gamma_4=\begin{psmallmatrix}
	& & 2& & &\\
	1&2&4&3&2&1
\end{psmallmatrix}.$$
The highest root is $\beta=\begin{psmallmatrix}
	& & 2& & &\\
	2& 3&4&3&2&1
\end{psmallmatrix}$.
Similarly as Lemma~\ref{lemma:E6_D4} we can deduce the following.
\begin{lemma}
	\begin{enumerate}[label=(\roman*)]
		\item The span $\langle \gamma_1,\dots,\gamma_4 \rangle\cap \Delta(\frakg,\frakt)$ is a root system of type $D_4$ with simple roots
		with $$\delta'_1=\delta_3, \qquad \delta'_2=\delta_1, \qquad\delta'_3=\begin{psmallmatrix}
			& & 1& & & \\
			0& 1&2&1&0&0
		\end{psmallmatrix}, \qquad \delta'_4=\begin{psmallmatrix}
			& & 1 & & \\
			0& 1&2&2&2 & 1
		\end{psmallmatrix}$$
		and the same highest root $\beta$, see \eqref{eq:D4}.
		\item We have $\delta_i \in \langle\gamma_1,\dots,\gamma_4, \beta \rangle_\ZZ$ if and only if $i=1,3$.
	\end{enumerate}
\end{lemma}

Then again by Lemma~\ref{lemma:root_vector_projection} and Proposition~\ref{prop:D4_projections} we have the following.
\begin{corollary}\label{cor:E7_projection}
	We have for simple roots $\alpha_j\in \Pi(\frakg,\frakh)$
	$$P_\beta(Z_{\alpha_j})=	\begin{cases}
		-\frac{1}{2}H_\beta & \text{if $j=1$,}\\
		\frac{1}{2}(X_\beta+X_{-\beta})& \text{if $j=3$,} \\
		0 & \text{if $j=2,4,5,6,7$.}
	\end{cases}$$
\end{corollary}

\subsection{The case $\frake_8$}
Let $\frakg=\frake_8$ and $\frakg_0=\frake_{8,4}$ the real form of real rank $4$. Then $\frakk=\sl(2,\CC)\oplus \frake_{7}$.
We choose the strongly orthogonal non-compact roots $\{\gamma_1,\dots,\gamma_4\}\in \Delta(\frakg,\frakt)$
in terms of $\Pi(\frakg,\frakt)=\{\delta_1,\dots,\delta_8\}$
by
$$
\gamma_1=\delta_8, \qquad \gamma_2=\begin{psmallmatrix}
		& & 1& & & &\\
	0&1&2&2&2&2 &1
\end{psmallmatrix}, \qquad \gamma_3=\begin{psmallmatrix}
	& & 2& & &\\
2&3&4&3&2&2 &1
\end{psmallmatrix}, \qquad \gamma_4=\begin{psmallmatrix}
	& & 3& & &\\
	2&4&6&5&4&2 &1
\end{psmallmatrix}.$$
The highest root is $\beta=\begin{psmallmatrix}
	& & 3& & & &\\
	2& 4&6&5&4&3&2
\end{psmallmatrix}$.
Similarly as Lemma~\ref{lemma:E6_D4} we can deduce the following.
\begin{lemma}
	\begin{enumerate}[label=(\roman*)]
		\item The span $\langle \gamma_1,\dots,\gamma_4 \rangle\cap \Delta(\frakg,\frakt)$ is a root system of type $D_4$ with simple roots
		with $$\delta'_1=\delta_7, \qquad \delta'_2=\delta_8, \qquad\delta'_3=\begin{psmallmatrix}
			& & 1& & & &\\
			0& 1&2&2&2&1&0
		\end{psmallmatrix}, \qquad \delta'_4=\begin{psmallmatrix}
			& & 2 & & &\\
			2& 3&4&3&2 & 1&0
		\end{psmallmatrix}$$
		and the same highest root $\beta$, see \eqref{eq:D4}.
		\item We have $\delta_i \in \langle\gamma_1,\dots,\gamma_4, \beta \rangle_\ZZ$ if and only if $i=7,8$.
	\end{enumerate}
\end{lemma}

By Lemma~\ref{lemma:root_vector_projection} and Proposition~\ref{prop:D4_projections} we have the following.
\begin{corollary}\label{cor:E8_projection}
	We have for simple roots $\alpha_j\in \Pi(\frakg,\frakh)$
	$$P_\beta(Z_{\alpha_j})=	\begin{cases}
		-\frac{1}{2}H_\beta & \text{if $j=8$,}\\
		\frac{1}{2}(X_\beta+X_{-\beta})& \text{if $j=7$,} \\
		0 & \text{if $j=1,2,3,4,5,6$.}
	\end{cases}$$
\end{corollary}

\subsection{The case $\frakf_4$}
Let $\frakg=\frakf_4$ and $\frakg_0=\frakf_{4,4}$ the real form of real rank $4$. Then $\frakk=\sl(2,\CC)\oplus \sp(3,\CC)$.
The strongly orthogonal non-compact roots $\{\gamma_1,\dots,\gamma_4\}\in \Delta(\frakg,\frakt)$ are chosen in terms
of $\Pi(\frakg,\frakt)=\{\delta_1,\dots,\delta_4\}$ as
$$
\gamma_1=\delta_1, \qquad \gamma_2=\begin{psmallmatrix}
1&2&2&0
\end{psmallmatrix}, \qquad \gamma_3=\begin{psmallmatrix}
1&2&2&2
\end{psmallmatrix}, \qquad \gamma_4=\begin{psmallmatrix}
1&2&4&2
\end{psmallmatrix}.$$
The highest root is $\beta=\begin{psmallmatrix}
2&3&4&2
\end{psmallmatrix}$.
Similarly as Lemma~\ref{lemma:E6_D4} we can deduce the following.
\begin{lemma}
	\begin{enumerate}[label=(\roman*)]
		\item The span $\langle \gamma_1,\dots,\gamma_4 \rangle\cap \Delta(\frakg,\frakt)$ is a root system of type $D_4$ with simple roots
		$$\delta'_1=\delta_2, \qquad \delta'_2=\delta_1, \qquad\delta'_3=\begin{psmallmatrix}
0&1&2&2
		\end{psmallmatrix}, \qquad \delta'_4=\begin{psmallmatrix}
0&1&2&0
		\end{psmallmatrix}$$
		and the same highest root $\beta$, see \eqref{eq:D4}.
		\item We have $\delta_i \in \langle\gamma_1,\dots,\gamma_4, \beta \rangle_\ZZ$ if and only if $i=1,2$.
	\end{enumerate}
\end{lemma}

Lemma~\ref{lemma:root_vector_projection} and
Proposition~\ref{prop:D4_projections} imply then the following.
\begin{corollary}\label{cor:F4_projection}
We have for simple roots $\alpha_j\in \Pi(\frakg,\frakh)$
	$$P_\beta(Z_{\alpha_j})=	\begin{cases}
		-\frac{1}{2}H_\beta & \text{if $j=1$,}\\
		\frac{1}{2}(X_\beta+X_{-\beta})& \text{if $j=2$,} \\
		0 & \text{if $j=3,4$.}
	\end{cases}\qquad $$
\end{corollary}

\subsection{The case $\frakg_2$}
Let $\frakg=\frakg_2$ and $\frakg_0=\frakg_{2,2}$ the real form of
real rank $2$. Let $\{\delta_1, \delta_2\}$ be the two simple
roots with $\delta_1$ the longer root.
We choose the strongly orthogonal non-compact roots $\{\gamma_1,\gamma_2\}\in \Delta(\frakg,\frakt)$ 
as $$
\gamma_1=\delta_1, \qquad \gamma_2=\delta_1+2\delta_2.$$
The highest root is $\beta=2\delta_1+3\delta_2$
and  $\frakk=\sl(2,\CC)_\beta\oplus \sl(2,\CC)$ with the first
summand $\sl(2,\CC)_\beta$ being determined by $\beta$.
Even though $\frakg_0$ is only of real rank $2$, we can still make use of 
Lemma~\ref{lemma:E6_D4} in the following way.
Consider $\so(8,\CC)$ with root system $D_4$ and simple roots $\delta'_1,\dots, \delta'_4$ and highest root $\beta'=\delta'_1+2\delta'_2+\delta'_3+\delta'_4$, see \eqref{eq:D4}.
Let $\partial:\so(8,\CC)\to \so(8,\CC)$ be the triality, acting on the rootsystem $D_4$ by circular the permutation $$\delta'_1\to\delta'_3\to\delta'_4\to \delta'_1.$$
Then we obtain an isomorphism $$T:\frakg_2\to\so(8,\CC)^\partial,$$
via
$$\delta_1\mapsto \delta'_2, \qquad \delta_2\mapsto \frac{1}{3}(\delta'_1+\delta'_3+\delta'_4),$$
and
$$X_{\delta_1}\mapsto X_{\delta'_2},\qquad X_{\delta_2}\mapsto X_{\delta'_1}+X_{\delta'_3}+X_{\delta'_4}.$$
In particular this maps $\beta$ to $\beta'$ and is compatible with our choices of strongly orthogonal roots as given in
\eqref{eq:D4_SOR}.
$$\gamma_1\mapsto \delta'_2=\varepsi_2-\varepsi_3, \qquad \gamma_2 \mapsto \frac{1}{3}((\varepsi_2+\varepsi_3)+(\varepsi_1-\varepsi_4)+(\varepsi_1+\varepsi_4)).$$
Moreover since $\so(4,4)$ is the split real form of $\so(8,\CC)$, we have that the restricted root system is also of type $D_4$, such that similarly $\so(4,4)^\partial=\frakg_{2,2}$. We refer the reader to \cite{Bae02} for more details on the construction of $\frakg_2$ via the triality of $D_4$.
\begin{lemma}\label{lemma:G2_projection}
	We have
	$$P_\beta(Z_{\alpha_j})=	\begin{cases}
		-\frac{1}{2}H_\beta & \text{if $j=1$,}\\
		\frac{3}{2}(X_\beta+X_{-\beta})& \text{if $j=2$,}
	\end{cases}\qquad
P_{\delta_2}(Z_{\alpha_j})=	\begin{cases}
	\frac{1}{2}H_{\delta_2} & \text{if $j=1$,}\\
	-\frac{1}{2}(X_{\delta_2}+X_{-\delta_2}) & \text{if $j=2$.}
\end{cases}$$
\end{lemma}
\begin{proof}
	By the discussion above it is clear that we have  $$P_\beta(Z_{\alpha_1})=T^{-1}(P_{\beta'}(Z_{\delta_2}))$$
	and 
	$$P_\beta(Z_{\alpha_2})=T^{-1}(P_{\beta'}(Z_{\alpha_1}+Z_{\alpha_3}+Z_{\alpha_4})),$$
	which implies the first statement by Proposition~\ref{prop:D4_projections}(1).
	Similarly we have
	$$P_{\delta_2}(Z_{\alpha_1})=T^{-1}\left(\frac{1}{3}\sum_{j=1,3,4}P_{\delta_j}(Z_{\alpha_2})\right),$$
	and  
	$$P_{\delta_2}(Z_{\alpha_2})=T^{-1}\left(\frac{1}{3}\sum_{j=1,3,4}P_{\delta_j}(Z_{\alpha_1}+Z_{\alpha_3}+Z_{\alpha_4})\right),$$
	which implies the second statement by Proposition~\ref{prop:D4_projections}(2).
\end{proof}
\section{Proof of the main theorems}\label{sec:proofs}
\subsection{The case $\frakg \neq \frakg_2$}
Let $\frakg\in \{\frake_6,\frake_7,\frake_8,\frakf_4\}$.
We can combine the results of the corollaries~\ref{cor:E6_projection}, \ref{cor:E7_projection}, \ref{cor:E8_projection}, \ref{cor:F4_projection} in the following way. 
\begin{lemma}
	Let $\alpha_1$ be the unique simple root connected to $-\beta$ in the affine Dynkin diagram of $\Delta(\frakg,\frakh)$ and $\alpha_2$ the unique simple root connected to $\alpha_1$ (see Table~\ref{tab:satake}).
	Then $\alpha_1,\alpha_2$ are real roots and we have
		$$P_\beta(Z_{\alpha_j})=	\begin{cases}
		-\frac{1}{2}H_\beta & \text{if $j=1$,}\\
		\frac{1}{2}(X_\beta+X_{-\beta})& \text{if $j=2$,}
	\end{cases}$$
and $P_\beta(Z_\alpha)=0$ for every other simple root $\alpha \in \Pi(\frakg,\frakh)$.
\end{lemma}
We recall some standard facts about $\sl(2,\CC)$-representations.

\begin{lemma}\label{lemma:sl(2)_action}
Let $\{H,E,F\}$ be a $\sl(2,\CC) $-triple. For every $\sl(2,\CC)$ representation $\Gamma_m$ there exits a weight basis ${v_{-m},v_{-m+2},\dots,v_m}$ of $\Gamma_m$, such that
	$$H.v_k=kv_k, \qquad E.v_k=(m-k)v_{k+2}, \qquad F.v_k=(m+k)v_{k-2}.$$
\end{lemma}
Let $\lambda\in \Lambda=\Lambda_{\fg}$
and  $V(\lambda)$ the corresponding irreducible representation.
Further let
$$\lambda_1:=\lambda(\alpha_1^\vee), \qquad\lambda_2:=\lambda(\alpha_2^\vee),$$
and for $d\in \ZZ_{\geq 0}$ let  $$\epsilon_i:=\begin{cases}
	0 & \text{if $\lambda_i-d\equiv 0 \pmod 2$,}\\ 
	1&\text{if $\lambda_i-d\equiv 1 \pmod 2$,} 
\end{cases} \qquad b_i:=\min\{\lambda_i,d-\epsilon_i\}.$$
\begin{theo+}\label{thm:mult_not_G2}
	Assume $\lambda|_{\frakh_\frakm}=0$ and $\lambda(\alpha^\vee)\equiv 0 \pmod 2$ for every real simple root $\alpha \neq \alpha_1,\alpha_2$ (such a root exists only in the case $\frakf_4$, see Table~\ref{tab:satake}). Then $V(\lambda)|_\frakk$ contains the representation $\Gamma_m\boxtimes \mathbf{1}_{\frakm_c}$ if and only if $m=2d$ is even and $b_1+b_2\geq d$.
	In this case it contains the representation with the multiplicity
	$$m(\lambda,d)=\left[\frac{b_1+b_2-d+2}{2}\right].$$
\end{theo+}
\begin{proof}
	By Corollary~\ref{cor:Kostant_quaternionic} we need to find the dimension of the space
	$$\{ v \in \Gamma_m|\, q_{\lambda,i}(P_\beta(Z_{\alpha_i}))v =0, \, \forall \alpha_i \in \Pi_n(\frakg,\frakh) \}.$$
	Let $(v_{-m},v_{-m+2},\dots, v_m)$ be an ordered weight basis
        of $\Gamma_m$ for the $\sl(2,\CC)$-triple
        $\{h=H_\beta, e=X_{\beta}, f=X_{-\beta}\}$ as in Lemma~\ref{lemma:sl(2)_action}.
	Since $-\frac{1}{2}H_{\beta}$ has integer eigenvalues if and only if $m$ is even, we can assume $m=2d$ to be even.
	Then 
	$Z_{\alpha_1}$ acts on the representation
        space as $\frac 12 h$
        and $Z_{\alpha_2}$ as
        $\frac 12 h'$, $h'=e+f$.
	Then the statement follows from Proposition~\ref{prop:joint_kernel}.
\end{proof}

\begin{remark}\label{remark:h_m_res_0}
	It is easy to deduce what the condition $\lambda|_{\frakh_\frakm}=0$ implies for the highest weight $\lambda$ from the Satake diagram in Table~\ref{tab:satake}. If $\frakg_0$ is split, i.e. if $\frakg=\frakg_2,\frakf_4$ there is no condition. Otherwise we have
	$$\lambda|_{\frakh_\frakm}=0 \Leftrightarrow \begin{cases}
		\lambda(\alpha_1^\vee)=\lambda(\alpha_6^\vee), \,\lambda(\alpha_3^\vee)=\lambda(\alpha_5^\vee)& \text{if $\frakg=\frake_6$,}\\
		\lambda(\alpha_2^\vee)=\lambda(\alpha_5^\vee)=\lambda(\alpha_7^\vee)=0& \text{if $\frakg=\frake_7$,}\\
			\lambda(\alpha_2^\vee)=\lambda(\alpha_3^\vee)=\lambda(\alpha_4^\vee)=\lambda(\alpha_5^\vee)=0& \text{if $\frakg=\frake_8$.}
	\end{cases}$$
\end{remark}

Let $\CC_\beta\subseteq\sl(2,\beta)$ be the one-dimensional subalgebra spanned by $H_\beta$ and consider the subalgebra $\frakl=\CC_\beta \oplus \frakm_c \subseteq \frakk$.
\begin{theo+}\label{thm:twistor_mult_not_G2}
		The representation $V(\lambda)$ contains a $\frakl$-fixed vector if and only if 
			$\lambda|_{\frakh_\frakm}=0$,
			$\lambda(\alpha^\vee)\equiv 0 \pmod 2$ for every simple root $\alpha \neq \alpha_1,\alpha_2$.
			In this case the space of $\frakl$-fixed vectors has dimension $m(\lambda)$, given by 
			$$
			m(\lambda)=\left[\frac{(\lambda_1+1)(\lambda_2+1)+1}{2}\right].
			$$
\end{theo+}
\begin{proof}
	We can apply Theorem~\ref{thm:mult_not_G2} and perform restriction in stages from $\frakg$ to $\frakk$, and then to $\frakl$. A representation $\Gamma_{2d}\boxtimes \mathbf{1}_{\frakm_c}$ of $\frakk$ contains the trivial representation $\mathbf{1}_\frakl$ of $\frakl$ with multiplicity one, as the $0$ weight space of $\Gamma_{2d}$. Hence we can sum up all multiplicities of $\Gamma_{2d}\boxtimes\mathbf{1}_{\frakm_c}$ occuring in $V(\lambda)|_\frakk$.
	Assume wlog that $\lambda_1\leq \lambda_2$. 
	Then for $d\leq \lambda_1$ we have
	$$m(\lambda,d)=\left[\frac{d-\epsilon_1-\epsilon_2+2}{2}\right],$$
	for $\lambda_1<d\leq \lambda_2$ 
	$$m(\lambda,d)=\left[\frac{\lambda_1-\epsilon_2+2}{2}\right],$$
	and for $\lambda_2< d \leq \lambda_1+\lambda_2$ 
$$	m(\lambda,d)=\left[\frac{\lambda_1+\lambda_2-d+2}{2}\right].$$
	
	First assume $\lambda_1\equiv \lambda_2\equiv 0 \pmod 2$. Then summing over $d$ in the three intervals $0\leq d \leq \lambda_1$, $\lambda_1<d\leq \lambda_2$, $\lambda_2<d \leq \lambda_1+\lambda_2$ and splitting into even and odd indices we have
	\begin{align*}
		m(\lambda)&=\sum_{d=0}^{\frac{\lambda_1}{2}}(d+1)+\sum_{d=0}^{\frac{\lambda_1}{2}-1}d+ \frac{\lambda_2-\lambda_1}{2}\left(\frac{\lambda_1}{2}+1\right)+
		\left(\frac{\lambda_2-\lambda_1}{2}-1\right)\frac{\lambda_1}{2}
		+\sum_{d=0}^{\frac{\lambda_1}{2}-1}(d+1)+\sum_{d=0}^{\frac{\lambda_1}{2}-1}(d+1)\\
		&=\frac{1}{8}\left((\lambda_1+2)(\lambda_1+4)+\lambda_1(\lambda_1-2)+2\lambda_1(\lambda_1+2)+2(\lambda_2-\lambda_1)(\lambda_1+2)+2(\lambda_2-\lambda_1-2)\lambda_1\right) \\
		&=\frac{1}{2}(\lambda_1+\lambda_2+\lambda_1\lambda_2+2)=\frac{1}{2}((\lambda_1+1)(\lambda_2+1)+1).
	\end{align*}
Similarly if $\lambda_1\equiv \lambda_2\equiv1 \pmod 2$ we obtain
	\begin{align*}
	m(\lambda)&=\sum_{d=0}^{\frac{\lambda_1-1}{2}}d+\sum_{d=0}^{\frac{\lambda_1-1}{2}}(d+1)+ 
	\left(\lambda_2-\lambda_1\right)\frac{\lambda_1+1}{2}
	+\sum_{d=0}^{\frac{\lambda_1-1}{2}}(d+1)+\sum_{d=0}^{\frac{\lambda_1-3}{2}}(d+1)\\
	&=\frac{1}{2}(\lambda_1+\lambda_2+\lambda_1\lambda_2+1)=\frac{1}{2}(\lambda_1+1)(\lambda_2+1).
\end{align*}
For $\lambda_1\equiv0 \pmod 2$ and $\lambda_2 \equiv1 \pmod 2$ we obtain
	\begin{align*}
	m(\lambda)&=\sum_{d=0}^{\frac{\lambda_1}{2}}d+\sum_{d=0}^{\frac{\lambda_1}{2}-1}(d+1)+ 
	\left(\frac{\lambda_2-\lambda_1-1}{2}\right)\frac{\lambda_1}{2}+\left(\frac{\lambda_2-\lambda_1+1}{2}\right)\left(\frac{\lambda_1}{2}+1\right)
	+\sum_{d=0}^{\frac{\lambda_1}{2}-1}(d+1)+\sum_{d=0}^{\frac{\lambda_1}{2}-1}(d+1)\\
	&=\frac{1}{2}(\lambda_1+\lambda_2+\lambda_1\lambda_2+1)=\frac{1}{2}(\lambda_1+1)(\lambda_2+1).
\end{align*}
The remaining case follows by symmetry.
\end{proof}

\subsection{The case $\frakg_2$}
We let $\lambda_i, \epsilon_i,b_i$, $i=1,2$ be as before.

\begin{theo+}\label{thm:mult_G2_short}
	The representation $V(\lambda)$ restricted to $\frakk$ contains the representation $ \mathbf{1}_{\beta}\boxtimes \Gamma_m$ if and only if $m=2d$ is even and $b_1+b_2\geq d$.
	In this case it contains the representation with the multiplicity
	$$m_{\alpha_2}(\lambda,d)=\left[\frac{b_1+b_2-d+2}{2}\right].$$
\end{theo+}
\begin{proof}
	Here $\Gamma_m$ is a representation of $\sl(2,\CC)_{\alpha_2}$ and by Theorem~\ref{thm:Kostant} the multiplicity is given by the dimension of the space
	$$\{ v \in \Gamma_m|\, q_{\lambda,i}(P_{\alpha_2}(Z_{\alpha_i}))v =0, \, i=1,2 \}.$$
	Let $(v_{-m},v_{-m+2},\dots v_m)$ be again a  weight basis for the triple $\{H_{\alpha_2},X_{\alpha_2},X_{-\alpha_2}\}$ as in Lemma~\ref{lemma:sl(2)_action}.
	Then by Lemma~\ref{lemma:G2_projection} we have that the statement follows in the same way from Proposition~\ref{prop:joint_kernel} as in the proof of Theorem~\ref{thm:mult_not_G2}.
\end{proof}
Now we define
$$\lambda_2'=\begin{cases}
	\left[\frac{\lambda_2}{3}\right] & \text{if $\lambda_2 \equiv 0,2 \pmod 3$,}\\
	\left[\frac{\lambda_2}{3}\right]-1 & \text{if $\lambda_2 \equiv 1 \pmod 3$,}
\end{cases}, $$
$$\epsilon_2':=\begin{cases}
	0 & \text{if $\lambda'_2-d\equiv 0 \pmod 2$,}\\ 
	1&\text{if $\lambda_2'-d\equiv 1 \pmod 2$,} 
\end{cases} \qquad b_2':=\min\{\lambda_2',d-\epsilon_2'\}.$$
\begin{theo+}\label{thm:mult_G2_long}
	The representation $V(\lambda)$ restricted to $\frakk$ contains the representation $ \Gamma_m\boxtimes \mathbf{1}_{\frakm_c}$ if and only if $m=2d$ is even and $b_1+b'_2\geq d$.
	In this case it contains the representation with the multiplicity
	$$m_{\beta}(\lambda,d)=\left[\frac{b_1+b'_2-d+2}{2}\right].$$
\end{theo+}
\begin{proof}
	Here $\Gamma_m$ is a representation of $\sl(2,\CC)_{\beta}$ and by  Corollary~\ref{cor:Kostant_quaternionic} the multiplicity is given by the dimension of the space
	$$\{ v \in \Gamma_m|\, q_{\lambda,i}(P_{\beta}(Z_{\alpha_i}))v =0, \, i=1,2 \}.$$
	Let $V=\text{Span}\{v_{-m},v_{-m+2},\dots v_m\}$ be the
        representation
        space with weight basis for
        the triple $\{h=H_{\beta}, e=X_{\beta}, f=X_{-\beta}\}$ as in Lemma~\ref{lemma:sl(2)_action}.
	Then by Lemma~\ref{lemma:G2_projection} we can again assume
        $m=2d$ to be even and we have that $P_\beta(Z_{\alpha_1})$
acts as  $-\frac 12  
h$ and         $P_\beta(Z_{\alpha_2})$ as
  $-\frac 32         h' =  -\frac 32    (e+f)$.
	We define the polynomial
	$$q'_{\lambda,2}(t)=\left(t-\frac{\lambda_2}{3}\right)\left(t-\frac{\lambda_2-2}{3}\right)\dots \left(t+\frac{\lambda_2}{3}\right),$$
	Such that $$\ker q_{\lambda,2}(Z_{\alpha_2})=\ker q'_{\lambda,2}(M_2(d)).$$
	It is easy to see that the product of the simple factors with integer zeroes of $q'_{\lambda,2}(t)$ is given by
	$$(t-\lambda_2')(t-\lambda_2'+2)\dots(t+\lambda_2')$$
	for
	$$\lambda_2'=\begin{cases}
		\left[\frac{\lambda_2}{3}\right] & \text{if $\lambda_2 \equiv 0,2 \pmod 3$,}\\
		\left[\frac{\lambda_2}{3}\right]-1 & \text{if $\lambda_2 \equiv 1 \pmod 3$,}
	\end{cases}.$$
Then the theorem follows in the same way as Theorem~\ref{thm:mult_not_G2} and Theorem~\ref{thm:mult_G2_short}.
\end{proof}

Let $\frakl_\beta:=\CC_\beta \oplus \sl(2,\CC)_{\alpha_2}$ and $\frakl_\alpha:=\sl(2,\CC)_\beta\oplus \CC_{\alpha_2},$ where $\CC_{\alpha_2}$ is the one-dimensional subalgebra generated by $H_{\alpha_2}$. In the same way as Theorem~\ref{thm:twistor_mult_not_G2} we obtain the following.
\begin{theo+}\label{thm:twistor_mult_G2}
	\begin{enumerate}
		\item 	The representation $V(\lambda)$ contains a $\frakl_\beta$-fixed vector if and only if $\lambda_2\neq 1$.
		In this case the space of $\frakl_\beta$-fixed vectors has dimension $m_\beta(\lambda)$, given by 
		$$
		m_\beta(\lambda)=\left[\frac{(\lambda_1+1)(\lambda'_2+1)+1}{2}\right].
		$$
		\item The representation $V(\lambda)$ contains a $\frakl_{\alpha_2}$-fixed vector.
		In this case the space of $\frakl_{\alpha_2}$-fixed vectors has dimension $m_{\alpha_2}(\lambda)$, given by 
		$$
		m_{\alpha_2}(\lambda)=\left[\frac{(\lambda_1+1)(\lambda_2+1)+1}{2}\right].
		$$
	\end{enumerate}
\end{theo+}

\section{The classical cases}\label{sec:classical}
The main results of this article concern exceptional quaternionic symmetric pairs. For completeness we now consider the classical quaternionic symmetric spaces
$$G_c/K=\begin{cases}
	{\rm U}(n+2)/({\rm U}(2)\times {\rm U}(n)) & n\geq 2,\\
	\SO(n+4)/(\SO(4)\times \SO(n))& n\geq 4, \\
	\Sp(n+1)/(\Sp(1)\times \Sp(n)) & n\geq 1,
\end{cases}$$
where we assume $K\subseteq G_c$ is embedded in diagonal blocks.
We recall the main results of \cite{Kna01} in this case and give the multiplicities for representations of representations contained in the $L^2$-spaces of the associated twistor spaces.
\begin{remark}
	The methods of this article can easily be adapted to the classical quaternionic cases. The available statements in the literature, e.g \cite{Kna01,Lee74,HTW05} and references therein, are already more general, such that we rather give a short exposition of some of the known results and specify them explicitly in the quaternionic setting. 
\end{remark}
\subsection{The unitary case}
Let $n\geq 2$ and $V(\lambda)$ be an irreducible representation of ${\rm U}(n+2)$ with highest weight $\lambda \in \ZZ^{n+2}$, given in the usual way with
$$\lambda_1\geq \lambda_2\geq \dots \geq \lambda_{n+2}.$$
\begin{theo+}[\cite{Kna01} Theorem~2.1]\label{thm:Knapp_U}
	The space $V(\lambda)^{{\rm U}(n)}$ of ${\rm U}(n)$-fixed vectors in $V(\lambda)$ is non-zero, if and only if $$\lambda_3=\lambda_4=\dots=\lambda_n=0$$ and $\lambda_2\geq 0, \lambda_{n+1}\leq 0$. As a ${\rm U}(2)$-representation, $V(\lambda)^{{\rm U}(n)}$ is equivalent to the tensor product of the ${\rm U}(2)$-representations with highest weights $(\lambda_1,\lambda_2)$ and $(\lambda_{n+1},\lambda_{n+2})$. 
\end{theo+}
We remark that the  ${\rm U}(2)$ tensor-product is multiplicity free and can be given by the classical formulas for the tensor product of $\SU(2)$-representations.
\begin{coro+}\label{cor:Knapp_U_mult}
	If $V(\lambda)^{{\rm U}(n)}$ is non-trivial, it decomposes as a ${\rm U}(2 )$-representation as
	$$
	\bigoplus_{k=0}^{\min(\lambda_1-\lambda_2,\lambda_{n+1}-\lambda_{n+2})}W(\lambda_1+\lambda_{n+1}-k,\lambda_2+\lambda_{n+2}+k),
	$$
	where $W(\nu_1,\nu_2)$, $\nu_1\geq \nu_2$ is the irreducible ${\rm U}(2)$ representation of highest weight $(\nu_1,\nu_2)$.
\end{coro+}
\begin{proof}We use some standard facts about representations of ${\rm U}(n)$ and $\SU(n)$. See e.g. \cite{GW09}.
	We have that for any ${\rm U}(2)$-highest weight $(\mu_1,\mu_2),$
	the restriction to $\SU(2)$ 
	$$W(\mu_1,\mu_2)|_{\SU(2)}=S^{\mu_1-\mu_2}(\CC^2)$$
	is given by the $(\mu_1-\mu_2+1)$-dimensional irreducible representation of $\SU(2)$.
	Hence as a $\SU(2)$-representation, we have
	$$W(\lambda_1,\lambda_2)|_{\SU(2)}\otimes W(\lambda_{n+1},\lambda_{n+2})|_{\SU(2)}=\bigoplus_{k=0}^{\min(\lambda_1-\lambda_2,\lambda_{n+1}-\lambda_{n+2})}S^{\lambda_1-\lambda_2+\lambda_{n+1}-\lambda_{n+2}-2k}(\CC^2).$$
	Now for every $z\in {\rm U}(1)$ we have that $z\mathbf{1}_2\in {\rm U}(2)$ is acting on $W(\mu_1,\mu_2)$ by $z^{\mu_1+\mu_2}$, such that $z\mathbf{1}_2$ acts on the tensor product
		$W(\lambda_1,\lambda_2)\otimes W(\lambda_{n+1},\lambda_{n+2})$ by
		$z^{\lambda_1+\lambda_2+\lambda_{n+1}+\lambda_{n+2}}$. Hence 
		$$S^{\lambda_1-\lambda_2+\lambda_{n+1}-\lambda_{n+2}-2k}(\CC^2)
		$$ uniquely extends to the ${\rm U}(2)$-representation $$W(\lambda_1+\lambda_{n+1}-k,\lambda_2+\lambda_{n+2}+k)$$ in the ${\rm U}(2)$-tensor product.
\end{proof}

Let $L={\rm U}(1)\times {\rm U}(n) \subseteq \SU(2)\times {\rm U}(n) \subseteq {\rm U}(2)\times {\rm U}(n).$
Consider the twistor space
$${\rm U}(n+2)/L.$$ From Theorem~\ref{thm:Knapp_U} and Corollary~\ref{cor:Knapp_U_mult} we can immediately deduce the following.
\begin{theo+}\label{thm:twistor_mult_U}
	Let $V(\lambda)$ be an irreducible representation of ${\rm U}(n+2)$ with highest weight $\lambda \in \ZZ^{n+2}$ as before.
	Then $V(\lambda)$ contains a non-zero $L$-fixed vector if and only if 
	\begin{enumerate}[label=(\roman*)]
		\item $\lambda_3=\lambda_4=\dots=\lambda_n=0$,
		\item $\lambda_2\geq 0$ and $\lambda_{n+1}\leq 0$,
		\item $\lambda_1-\lambda_2+\lambda_{n+1}-\lambda_{n+2}\equiv 0 \pmod 2$.
	\end{enumerate}
	In this case the space of $L$-fixed vectors has dimension
	$$m(\lambda)=\min(\lambda_1-\lambda_2,\lambda_{n+1}-\lambda_{n+2})+1.$$
\end{theo+}

\subsection{The orthogonal case}
Let $n\geq 4$ and $V(\lambda)$ be the irreducible representation of $\SO(n+4)$ with highest weight $\lambda\in \ZZ^{[\frac{n+4}{2}]}$, given in the usual way with
\begin{equation}\label{eq:SO_weights}
	\begin{cases}
	\lambda_1\geq \lambda_2\geq \dots \geq\abs{ \lambda_{\frac{n+4}{2}} }& \text{if $n$ is even,}\\
	\lambda_1 \geq \lambda_2 \geq \dots \geq \lambda_{\frac{n+3}{2}}\geq 0 & \text{if $n$ is odd.}
\end{cases}
\end{equation}

\begin{theo+}[\cite{Kna01} Theorem~3.1]\label{thm:Knapp_O}
	The space $V(\lambda)^{\SO(n)}$ of $\SO(n)$-fixed vectors in $V(\lambda)$ is non-zero, if and only if $$\lambda_5=\lambda_6=\dots=\lambda_{[\frac{n+4}{2}]}=0.$$ As a $\SO(4)$-representation, $V(\lambda)^{\SO(n)}$ is equivalent to the restriction of irreducible the ${\rm U}(4)$-representation
	with highest weight $(\lambda_1,\lambda_2,\lambda_3,\abs{\lambda_4})$ to $\SO(4)$.
\end{theo+}
We can find the relevant representations in the restriction from ${\rm U}(4)$ to $\SO(4)$
using Theorem~\ref{thm:Kostant}. Let $\lambda'=(\lambda_1,\lambda_2,\lambda_3, \lambda_4)\in \ZZ^4$ with
$$\lambda_1\geq \lambda_2\geq \lambda_3\geq \lambda_4$$
be an integral dominant weight
of ${\rm U}(4)$. We define $$\lambda_1':=\min(\lambda_1-\lambda_2,\lambda_3-\lambda_4), \qquad \lambda_2'=\lambda_2-\lambda_3.$$
Further let again for $d\in \ZZ_{\geq 0}$ let  $$\epsilon_i:=\begin{cases}
	0 & \text{if $\lambda'_i-d\equiv 0 \pmod 2$,}\\ 
	1&\text{if $\lambda'_i-d\equiv 1 \pmod 2$,} 
\end{cases} \qquad b_i:=\min\{\lambda'_i,d-\epsilon_i\}.$$
Recall that $\SO(4)\cong {\SU(2)\times \SU(2)}/\{\pm(\mathbf1,\mathbf{1})\}$ such that irreducible representations of $\SO(4)$ can be given by pairs of representations of $\SU(2)$,
$S^\mu(\CC^2)\boxtimes S^\nu(\CC^2)$ where $\mu\equiv \nu \pmod 2$.
	\begin{prop+}\label{prop:U(4)_branching}
	Let $W(\lambda')$ be the irreducibel representation of ${\rm U}(4)$ with highest weight $\lambda'$ as before. Then $W(\lambda')|_{\SO(4)}$ contains the $\SO(4)$-representation $\mathbf{1}_{\SU(2)}\boxtimes S^m(\CC^2)$ (and equivalently the representation $S^m(\CC^2)\boxtimes \mathbf{1}_{\SU(2)}$) if and only if
	\begin{enumerate}[label=(\roman*)]
		\item $m=2d$ is even,
		\item $\lambda_1-\lambda_2\equiv \lambda_3-\lambda_4 \pmod 2$,
		\item $b_1+b_2\geq d$.
	\end{enumerate}
In this case it contains the representation with the multiplicity
$$m(\lambda',d)=\left[\frac{b_1+b_2-d+2}{2}\right].$$
\end{prop+}
\begin{proof}
	Let $\frakg_0=\sl(4,\RR)$, $\frakk_0=\so(4)$ such that $(\frakg,\frakk)$ is the complex symmetric pair $(\sl(4,\CC),\so(4,\CC))$.
	Then $\so(4,\CC)$ is given by the two $\sl(2,\CC)$-triples
	$\{H_{1},E_{1},F_{1}\}$, $\{H_2,E_2,F_2\}$ with
	$$H_1=\begin{pmatrix}
		0& i &0&0\\
		-i&0 &0&0 \\
		0&0&0&-i\\
		0&0&i&0
	\end{pmatrix},\qquad 
E_1=\frac{1}{2}\begin{pmatrix}
	0&0&1&i \\
	0&0&-i&1\\
	-1&i&0&0\\
	-i&-1&0&0
\end{pmatrix},\qquad F_1=\frac{1}{2}\begin{pmatrix}
0&0&-1&i \\
0&0&-i&-1\\
1&i&0&0\\
-i&1&0&0
\end{pmatrix},
$$
	$$H_2=\begin{pmatrix}
	0& i &0&0\\
	-i&0 &0&0 \\
	0&0&0&i\\
	0&0&-i&0
\end{pmatrix},\qquad 
E_2=\frac{1}{2}\begin{pmatrix}
	0&0&1&-i \\
	0&0&-i&-1\\
	-1&i&0&0\\
	i&1&0&0
\end{pmatrix},\qquad F_2=\frac{1}{2}\begin{pmatrix}
	0&0&-1&-i \\
	0&0&-i&1\\
	1&i&0&0\\
	i&-1&0&0
\end{pmatrix}.
$$
Choosing the split Cartan $\frakh\subseteq\frakg$ of diagonal matrices and simple roots such that positive root spaces are upper triangular matrices we find
$$Z_{\alpha_1}=\frac{1}{2}(H_1+H_2), \qquad Z_{\alpha_2}=-\frac{1}{2}(E_1+F_1+E_2+F_2),\qquad Z_{\alpha_3}=-\frac{1}{2}(H_1-H_2)$$
for the simple roots $\alpha_1,\alpha_2,\alpha_3$ of $(\frakg,\frakh)$.
Let $\varpi_1,\varpi_2,\varpi_3$ be the corresponding fundamental weights.
The restriction of the integral dominant weight $\lambda'$ of $\gl(4,\CC)$ resp. $\fraku(4)$ to $\frakg$ gives the integral dominant weight
$$(\lambda_1-\lambda_2)\varpi_1+(\lambda_2-\lambda_3)\varpi_2+(\lambda_3-\lambda_4)\varpi_3,$$
see e.g. \cite[Theorem~5.5.22]{GW09}. In particular $\frakg_0$ is split, such that all simple roots are real. Then the statement follows in the same way as Theorems~\ref{thm:mult_not_G2}\ref{thm:mult_G2_short} and \ref{thm:mult_G2_long} from Theorem~\ref{thm:Kostant} and Proposition~\ref{prop:joint_kernel}.
\end{proof}

\begin{remark}
	Proposition~\ref{prop:U(4)_branching} together with Theorem~\ref{thm:Knapp_O} now gives a full generalization of the Cartan--Helgason Theorem for the quaternionic compact symmetric space $\SO(n+4)/(\SO(4)\times \SO(n))$, $n\geq 4$.
	The branching from ${\rm U}(n)$ to $\SO(n)$ can also be deduced from \cite[2.4.1]{HTW05} in terms of sums of products of Littlewood--Richardson coefficients. Yet our formula is more explicit in this special case.
\end{remark}
Let $L={\rm U}(1)\times \SO(n) \subseteq \SO(4)\times \SO(n)$, where we consider ${\rm U}(1)$ to be given by equivalence classes of the the circle contained in the left copy of $\SU(2)\times \SU(2)\to \SO(4)$.
From Theorem~\ref{thm:Knapp_O} and Proposition~\ref{prop:U(4)_branching} we immediately obtain in the same way as Theorems~\ref{thm:twistor_mult_not_G2},\ref{thm:twistor_mult_G2} the following.
\begin{theo+}\label{thm:twistor_mult_O}
	Let $V(\lambda)$ be the irreducible representation of $\SO(n+4)$ with highest weight $\lambda \in \ZZ^{[\frac{n+4}{2}]}$ as before.
	Then $V(\lambda)$ contains a non-zero $L$-fixed vector if and only if 
	\begin{enumerate}[label=(\roman*)]
		\item $\lambda_5=\lambda_6=\dots=\lambda_{[\frac{n+4}{2}]}=0$,
		\item $\lambda_1-\lambda_2+\lambda_{3}-\abs{\lambda_{4}}\equiv 0 \pmod 2$.
	\end{enumerate}
	In this case the space of $L$-fixed vectors has dimension
	$$m(\lambda)=\left[\frac{(\lambda_1'+1)(\lambda_2'+1)+1}{2}\right],$$
	where $\lambda_1'=\min(\lambda_1-\lambda_2,\lambda_3-\abs{\lambda_4})$ and
	$\lambda_2'=\lambda_2-\lambda_3$.
\end{theo+}

\subsection{The symplectic case}
Let $n\geq 1$ and $V(\lambda)$ be the irreducible representation of $\Sp(n+1)$ with highest weight $\lambda \in \ZZ^{n+1}$, given in the usual way with
$$\lambda_1\geq \lambda_2\geq \dots \geq \lambda_{n+1}\geq 0.$$
\begin{theo+}[\cite{Kna01} Theorem~4.1]\label{thm:Knapp_Sp}
	The space $V(\lambda)^{\Sp(n)}$ of $\Sp(n)$-fixed vectors in $V(\lambda)$ is non-zero, if and only if $$\lambda_3=\lambda_4=\dots=\lambda_{n+1}=0.$$ As a $\Sp(1)\cong\SU(2)$-representation, $V(\lambda)^{\Sp(n)}$ is irreducible and equivalent to the irreducible $(\lambda_1-\lambda_2+1)$-dimensional representation 
	$S^{\lambda_1-\lambda_2}(\CC^2)$ of $\SU(2)$.
\end{theo+}
\begin{remark}
	For the symplectic case there is a complete branching law available for the restriction of arbitrary $\Sp(n+1)$-representations to $\Sp(1)\times \Sp(n)$, see e.g. \cite{Lee74}.
\end{remark}
Now let $L={\rm U}(1)\times \Sp(n)\subseteq \Sp(1)\times \Sp(n)$.
We immediately obtain the following.
\begin{coro+}\label{cor:twistor_mult_Sp}
	Let $V(\lambda)$ be the irreducible representation of $\Sp(n+1)$ with highest weight $\lambda \in \ZZ^{n+1}$ as before.
	Then $V(\lambda)$ contains a non-zero $L$-fixed vector if and only if 
	\begin{enumerate}[label=(\roman*)]
		\item $\lambda_2=\lambda_3=\dots=\lambda_{n+1}=0,$
		\item $\lambda_1-\lambda_2\equiv 0 \pmod 2$.
	\end{enumerate}
	In this case the space of $L$-fixed vectors is one-dimensional.
\end{coro+}

\begin{rema+}\label{final-rem}
  We remark finally that for each quaternionic compact
  symmetric space $G_c/K$ with
  the complexified Lie algebra
  of $K$ being $\fk=\mathfrak{sl}(2, \mathbb C)_{\beta} +\fm_c$
  there is a twistor space $G/L$
  as a complex manifold
 fibered over $G_c/K$  with fiber
 being the projective space $K/L=
\SU(2)_\beta/{\rm U}(1)_\beta= \mathbb P^1$,
 where   $K'$ is locally ${\rm U}(1)_\beta\times M_c$.
 Our results above can be formulated
 as an irreducible decomposition
 for the space $L^2(G_c/{\rm U}(1)_\beta\times M_c, \chi_l)$
 for a line bundle defined by an even character of ${\rm U}(1)_\beta$
 over $G_c/L$, the differential $d\chi_l$
 being  $d\chi_l=l\beta$.
 This decomposition is important in studying
 Heisenberg parabolically induced
 representations \cite{Zha23}
 and we plan to pursue this study in
 the future.
\end{rema+}

\appendix
\section{Dimension
formula for joint kernels
 non-commuting operators}\label{app:linear_algebra}
We shall compute the dimension
of the joint kernel of two
non-commuting semisimple elements
on $V=\mathbb C^{2d+1}$. The result might be done by using 
explicit formulas for the eigenvectors on $h'$
obtained in \cite[Theorem~1]{CW08}
and may require much detailed linear algebra computations. 
We shall use here again the representation theory 
of $SL(2, \mathbb C)$.

Let $\mathfrak{s}=\mathfrak{sl}_{2}(\mathbb{C})$ with a standard $\mathfrak{sl}_{2}$-triple $\{h,e,f\}$, i.e.,  $[h,e]=2e$, 
$[h,f]=2f$, $[e,f]=h$, realized explicitly as 
\[h=\left(\begin{array}{cc}1&\\&-1\\\end{array}\right),\quad 
e=\left(\begin{array}{cc}&1\\0&\\\end{array}\right),\quad\ f=\left(\begin{array}{cc}&0\\1&\\\end{array}\right).\] 
Let $h'=e+f$. Then $h'$ is
the Cartan element of another $\mathfrak{sl}_{2}$-triple. 
We introduce also the $SL(2, \mathbb C)$-elements
$A:=\Exp(\frac{\pi\mathbf{i}}{2}h)$ and $B:=\Exp(\frac{\pi\mathbf{i}}{2}h')$ in $\SL_{2}(\mathbb{C})$; in the above realization, 
\[A=\left(\begin{array}{cc}\mathbf{i}&\\&-\mathbf{i}\\\end{array}\right), \quad 
B=\left(\begin{array}{cc}&
                           \mathbf{i}\\\mathbf{i}&\\\end{array}\right), 
                       \Gamma (A)f(z_1, z_2)= f(iz_1, -iz_2), \, 
                       \Gamma (B)f(z_1, z_2)= f(iz_2, iz_2). 
                     \]
They  generate the finite 
quaternionic group $Q_{8}$  as a subgroup of $\SL_{2}(\mathbb{C})$. 
                     The adjoint actions of $A$ on $h'$
and $B$ on $\{h, e, f\}$ are given by 
$$\Ad(A)h'=-h', 
\Ad(B)e=f, \Ad(B)f=e, \Ad(B)h=-h.$$

Recall our notation 
 $\Gamma:=\Gamma_{2d}$ ($d\in\mathbb{Z}_{>0}$),  the irreducible representation of 
 $\mathfrak{sl}(2)$ on $V=\mathbb C^{2d+1}$.
The Lie algebra and Lie group actions 
will all be written as $\Gamma$.
We realize $V$ as the space of
homogeneous polynomials $f(z_1, z_2)$ of
degree $2d$ with $\Gamma(g)f(z_1, z_2)= f((z_1, z_2)g)$,
$g\in SL(2, \mathbb C)$.
Since $\Gamma(-I)=1$, the image
of $Q_8$ under $\Gamma$ is the abelian group $\Gamma(Q_8)= \mathbb Z_{2}\times 
\mathbb Z_{2}$.

Under the above realization  we have $\Gamma(h)=z_1\partial_1 
-z_2\partial_z$ and 
$\Gamma(h')=z_1\partial_2+
z_2\partial_1$, and their eigenvectors
are 
$$e_k= z_1^{d+k}z_2^{d-k}, e_k'=
(z_1+z_2)^{d+k}(z_1-z_2)^{d-k}, \quad -d\le k\le k,
$$
with
$$
\Gamma(h) e_k=2k e_k,  \Gamma(h') e_k'=2k e_k'.
$$
The bases diagonalise also
the action of $\Gamma(Q_8)= \mathbb Z_{2}\times 
\mathbb Z_{2}$.  Put, for $i=0, 1$, 
\[V_{i}=\{v\in V: \Gamma(A)v=(-1)^iv\}, \quad 
V'_{i}=\{v\in V: \Gamma(B)v=(-1)^iv\}.\]
Then $V=V_0\oplus V_1 = V_0'  \oplus V_1' 
= V_{0, 0} \oplus  V_{0, 1} \oplus 
V_{1, 0} \oplus 
 V_{1, 1}, 
 $ where 
 $V_{i, j} =V_i\cap V_j'$.  Explicitly
we have
  \begin{equation}
    \label{AB-act}
V_i=\text{Span}\{e_k, k\equiv i \pmod 2\}, \quad 
V_i'=\text{Span}\{e_k^{\prime}, k\equiv i \pmod 2\}, \quad i=0, 1.     
  \end{equation}

Let  $(a_{k,l})_{-d\le k, l\le l}$ be the matrix for the change of bases,
$$
e_k= \sum_l a_{k,l} e_l'.
$$
Then it is of full rank $2d+1$, and 
satisfies
\begin{equation}\label{sym-a-kl}
a_{k,l}= (-1)^{d+l} a_{k, -l} = 
  (-1)^{d+l} a_{-k, l}
  \end{equation}
by the formula (\ref{AB-act})
for the actions of $\Gamma(A)$
and  $\Gamma(B)$.

 Write $p_{b}(x)$ the polynomial of degree $b+1$, 
 \[p_{b}(x)=\prod_{0\leq j\leq b}(x-2b+4j).\]
 Let $0\leq b, b'\leq d$. Consider 
 the operators $P:=p_{b}(\Gamma(h))=\Gamma(p_b(h))$ and $P':=p_{b'}(\Gamma(h'))=\Gamma(p_{b'}(h'))$
 and their joint kernel on $V$,
\[V_{b, b'}=\ker P
  \cap \ker P',\]
where $\ker P$ and $\ker P'$ can be found explicitly as
\begin{equation}\label{kerP-P}
\ker P=\text{Span}\{z_1^{d+k}z_2^{d-k}; 
-b\le k \le b,  k\equiv b \pmod 2\}
\end{equation}
and
\begin{equation}\label{kerP-P-1}
\ker P'=\text{Span}\{(z_1+z_2)^{d+k}(z_1-z_2)^{d-k}; 
-b'\le k \le b', k\equiv b' \pmod 2\}.
\end{equation}

 \begin{prop+}   \label{prop:joint_kernel}
The following dimension formula holds
 \[\dim V_{b,b'}=\lfloor\frac{b+b'-d+2}{2}\rfloor\] if $b+b'\geq d$; and 
$\dim V_{b,b'}=0$ otherwise.  
\end{prop+}

This will be proved by studying  the actions of $A$ and $B$ on the spaces 
$\ker P$ and 
$\ker P'$.

\begin{lemma}  \label{L:dim1} 
It holds
$V_{b, b'}
\subseteq V_{i, i'} $
where $b\equiv i=0, 1 $ and $b'\equiv i'=0, 1 \pmod 2$.
\end{lemma} 
This is obtained by a straightforward computation.

\begin{lemma}\label{L:dim3}  
If $b+b'\leq d-1$, then $V_{b,b'}=0$. 
\end{lemma} 

\begin{proof}
By (\ref{kerP-P})   $\ker P$  has a basis $\{v_{0},\Gamma(e)^{j}v_{0}, 
\Gamma(f)^{j}v_{0}:1\leq j\leq b\}$ for some $v_0$. Expand the product \[
P'=p_{b'}(\Gamma(h'))=\prod_{0\leq j\leq b'}\Gamma(e+f-2b'+4j)\] as a sum of 
$\Gamma(f^{i}h^{j}e^{k})$, with $f^{i}h^{j}e^{k}$ ($i,jk\geq 0$) 
being a basis of the universal enveloping algebra $\mathcal{U}(\mathfrak{sl}_{2}(\mathbb{C}))$. The 
highest term is $\Gamma(e^{b'+1})$ under the action of $\Gamma(h)$. But as $b+(b'+1)\leq d$ it follows that $P'$ acts on $\ker P$ 
as an injective linear map. Thus, $V_{b,b'}=0$. 
\end{proof}

\begin{proof} (Proposition \ref{prop:joint_kernel}) If $b+b'\le d-1$ then this is proved in the previous Lemma. We assume $ b+b'\ge d$.
The key idea is to use Lemma \ref{L:dim1} and to solve the system of equations
for $\ker P\cap \ker P'$ along with the actions of $A$ and $B$. We prove only
the case that $d, b, b'$ are even. Now 
$\ker P\cap \ker P' 
\subset (\ker P\cap V_0)$, $\ker P\cap \ker P' 
\subset (\ker P'\cap V_0')$,  the latter spaces being  of
dimensions $b+1$
and $b'+1$ respectively.
The space
$\ker P\cap \ker P' 
= (\ker P\cap V_0)
\cap (\ker P'\cap V_0')
$,
and that an element $f=\sum_{k} x_k e_k\in (\ker P\cap V_0)$
is also in $(\ker P'\cap V_0')$
is equivalent to $x_k= x_{-k}$ and
$$
\sum_k a_{kl}
x_k=0, 
$$
for $l\notin \{l'; -b'\le l'\le b', l'\equiv b' \pmod 2\}$.
But due to symmetry condition
$x_k= x_{-k}$ and (\ref{sym-a-kl})
this becomes a system of equations
for 
$$
l\notin \{l'; 0\le l'\le b', l\equiv b' \pmod 2\}$$
with $\frac b2 +1$ variables
$$
x_k, 
0\le k\le b, k\equiv b \pmod 2.
$$
The
corresponding
matrix
$(a_{k,l})$ for
$(k, l)$ as above is of
full rank $\frac{d}2+1 -
\frac{b'}2+1=\frac{d-b'}2  \le \frac{b}2 +1
$. Thus the dimension of the solution space is 
$$
\frac b2 +1 -\frac{d-b'}2 =
\frac{b+b'-d+2}2.
$$
\end{proof}

\section{Tables}\label{app:tables}
\begin{table}[h!]
	\begin{center}
		\begin{tabular}{|l|l|l|l|}
			\hline
			$\frakg$ & $\frakk$ & $\frakg_0$ & $\frakk_0$ \\
			\hline
			$\frake_6$ & $ \sl(2,\CC)\oplus\sl(6,\CC)$ & $\frake_{6,4}$ & $\su(2)\oplus \su(6)$  \\
			\hline
			$\frake_7$ & $\sl(2,\CC)\oplus \so(12,\CC)$ & $\frake_{7,4}$ & $ \su(2)\oplus\so(12)$ \\
			\hline
			$\frake_8$ & $\sl(2,\CC)\oplus\frake_7$ & $\frake_{8,4}$& $\su(2)\oplus \frake_{7,0}$ \\ 
			\hline  
			$\frakf_4$ &  $\sl(2,\CC)\oplus \sp(3,\CC) $& $\frakf_{4,4}$ & $\su(2)\oplus \sp(3)$ \\ 
			\hline
			$\frakg_2$ & $\sl(2,\CC)\oplus\sl(2,\CC)$ & $\frakg_{2,2}$ & $\su(2)\oplus \su(2)$ \\ 
			\hline
		\end{tabular}
		\vskip0.20cm
		\caption
		{Exceptional Lie algebras $\frakg$ with quaternionic real form $\frakg_0$ and subalgebras $\frakk$ and their compact real form $\frakk_0$.}
		\label{tab:2}
	\end{center}
\end{table}

\begin{table}[h!]
	\begin{center}
		\begin{tabular}{|l|l|l|l|}
			\hline
			$\frakg_0$ &Affine Dynkin diagram& Satake diagram & Real simple roots\\
			\hline
			$\frake_{6,4}$ & $\dynkin[edge length=.75cm,labels={-\beta,1,2, 3, 4,5, 6}]E[1]6$& $\dynkin[edge length=.75cm,
			involutions={16;35}]E{oooooo}$& $2,4$ \\
			\hline
			$\frake_{7,4}$ &$\dynkin[edge length=.75cm,labels={-\beta,1,2, 3, 4,5, 6,7}]E[1]7$ &$ \dynkin[edge length=.75cm]E{VI}$ & $1,3$ \\
			\hline
			$\frake_{8,4}$& $\dynkin[edge length=.75cm,labels={-\beta,1,2,3, 4,5, 6,7,8}]E[1]8$&$\dynkin[edge length=.75cm]E{IX}$ & $7,8$ \\
			\hline
			$\frakf_{4,4}$ & $\dynkin[edge length=.75cm,labels={-\beta,1,2, 3, 4}]F[1]4$&$\dynkin[edge length=.75cm]FI$ & $1,2,3,4$ \\
			\hline
			$\frakg_{2,2}$ &$\dynkin[edge length=.75cm,labels={-\beta,1,2}]G[1]2$ &$\dynkin[edge length=.75cm]GI$ & $1,2$\\
			\hline
		\end{tabular}
		\vskip0.20cm
		\caption
		{Affine Dynkin diagrams with labeling of simple roots $\alpha_i$ and lowest root $-\beta$, Satake diagrams and real simple roots of  quaternionic real forms of exceptional simple Lie algebras (see \cite{Hel78}).}
		\label{tab:satake}
	\end{center}
\end{table}

\newpage
\bibliographystyle{amsplain}

\end{document}